\documentclass[11pt]{article}
\usepackage{psfrag,amsmath}
\usepackage{amsthm}
\usepackage{amssymb}
\usepackage{mathrsfs}
\usepackage{enumerate}
\usepackage{amssymb,a4wide,latexsym,makeidx,epsfig,fleqn}
\usepackage{floatrow}
\usepackage{makecell}
\usepackage{geometry}
\usepackage{enumerate}
\usepackage{graphicx}
\usepackage{multirow}
\usepackage{color}
\usepackage{array}
\usepackage{subfigure}
 \usepackage{float}
 \usepackage{cases}
\usepackage{comment}
\usepackage[utf8]{inputenc}
\usepackage[T1]{fontenc}

\usepackage[titletoc]{appendix}

\usepackage{longtable}

\usepackage{titlesec}
\titleformat{\chapter}[display]
{\normalfont\Large\bfseries}{\thechapter}{11pt}{\Large}
\titleformat{\section}
{\normalfont\large\bfseries}{\thesection}{11pt}{\large}
\titlespacing*{\chapter}{0pt}{0pt}{15pt} 
\titlespacing*{\section}{0pt}{3.5ex plus 1ex minus .2ex}{2.3ex plus .2ex}
\usepackage[titletoc]{appendix}

\usepackage{indentfirst}
\newtheorem{theorem}{Theorem}[section]

\newtheorem{definition}[theorem]{Definition}
\newtheorem{lemma}[theorem]{Lemma}

\newtheorem{problem}[theorem]{Problem}

\begin{document}

	\textwidth 150mm \textheight 225mm
\title{The $A_\alpha$ spectral radius of graphs with given independence number $n-4$ \thanks{Supported by the National Natural Science Foundation of China (No. 11871398).} }
\author{{Xichan Liu\textsuperscript{a,b}, Ligong Wang\textsuperscript{a,b,}\footnote{Corresponding author.}}\\
	{\small \textsuperscript{a} School of Mathematics and Statistics
	}\\
	{\small Northwestern Polytechnical University, Xi'an, Shaanxi 710129, P.R. China.}\\
	{\small \textsuperscript{b} Xi'an-Budapest Joint Research Center for Combinatorics}\\
	{\small  Northwestern Polytechnical University, Xi'an, Shaanxi 710129, P.R. China.}\\
	{\small E-mail: xichanliu1998@163.com, lgwangmath@163.com}}
\date{}
\maketitle

\begin{center}
	\begin{minipage}{135mm}
		\vskip 0.3cm
		\begin{center}
			{\small {\bf Abstract}}
		\end{center}
		{\small    Let $G$ be a graph with adjacency matrix $A(G)$  and degree diagonal matrix $D (G)$. In 2017, Nikiforov [Appl. Anal. Discrete Math., 11 (2017) 81--107]  defined the matrix $A_\alpha(G) = \alpha D(G) + (1-\alpha)A(G)$ for any real $\alpha\in[0,1]$. The largest eigenvalue of $A(G)$ is called the  spectral radius of $G$, while the largest eigenvalue of $A_\alpha(G)$ is called the $A_\alpha$ spectral radius of $G$. Let $\mathcal{G}_{n,i}$ be the set of graphs of order $n$ with independence number $i$. Recently, for all graphs in $\mathcal{G}_{n,i}$ having the minimum or the maximum $A$, $Q$ and $A_\alpha$ spectral radius  where $i\in\{1,2,\lfloor\frac{n}{2}\rfloor\,\lceil\frac{n}{2}\rceil+1,n-3,n-2,n-1\}$, there are some results have been given by Xu, Li and Sun et al., respectively.  In 2021, Luo and Guo [Discrete Math., 345 (2022) 112778] determined all graphs in $\mathcal{G}_{n,n-4}$ having the minimum spectral radius. In this paper, we  characterize the graphs in $\mathcal{G}_{n,n-4}$ having the minimum and the maximum $A_\alpha$ spectral radius  for $\alpha\in[\frac{1}{2},1)$, respectively.
			 }
	
 \end{minipage}
\end{center}

\section{Introduction }
\label{sec:ch-introduction}
\textwidth 150mm \textheight 225mm
In this paper, we only consider simple and connected graphs. Let $G$ be a graph with vertex set $V(G) = \{v_1, v_2,\ldots, v_n\}$ and edge set $E(G)$. The order of $G$ is $n = |V(G)|$ and the size of $G$ is $m = |E(G)|$. If two vertices  $u$ and $v$ in $G$ are joined by an edge $uv$, then we say that the two vertices $u$, $v$ are adjacent or neighbors and the edge $uv$ is incident with $u$ and $v$. 
   The neighbor set of a vertex $v$ in $G$ is  denoted by $N_G(v)$. The degree of a vertex $v$ in $G$ is the number of edges incident with $v$, denoted by $d_G(v)$. The minimum degree  and the maximum degree of $G$ are denoted by $\delta(G)$ and $\Delta(G)$, respectively.  A subset $S\subseteq V(G)$  is called an independent set of $G$ if no two vertices of $S$ are adjacent in $G$. The independence number $i(G)$ of $G$ is the maximum cardinality of an independent set in $G$.   An internal path of $G$ is a path $v_{0}v_{1}v_{2} \ldots v_{s}$ $(s \geq 1)$  with  $d_{G}(v_0) \geq 3$, $d_{G}(v_s) \geq 3$ and $d_{G}(v_i) = 2$ $( i=1,2,\ldots,s-1)$. The complement of $G$ is a graph $G^c$ with $V(G^c)=V(G)$ and any two distinct vertices are adjacent in $G^c$ if and only if they are non-adjacent in $G$. For two vertices $u$ and $v$ in $G$,	let $G-u$ be a graph obtained from $G$ by deleting $u$ and all edges incident with $u$. Let $G-uv$ (resp. $G+uv$) be a graph obtained from $G$ by deleting (resp. adding) the edge $uv$. Let $G_1$,  $G_2$ be two vertex-disjoint graphs, we use $G_1\vee G_2$ to denote their join graph, whose vertex set is $V(G_1)\cup V(G_2)$, and its edge set equals $E(G_1)\cup E(G_2)\cup \{uv:u\in V(G_1),v\in V(G_2)\}$. A path, a cycle, a complete graph of order $n$ are denoted by $P_n$,  $C_n$, $K_n$, respectively.  A vertex of a graph is said to be major if its degree is at least 3. Let $T_{a,b,c}$ be a tree with only one major vertex $u$, such that $T_{a,b,c}-u = P_a \cup P_b \cup P_c$.
  Unless otherwise stated, we follow the traditional notations and terminologies in $\cite{b-2}$.

The adjacency matrix of $G$ is an $n\times n$ 0-1 matrix, denoted by $A(G)=[a_{ij}]$, where $a_{ij}=1$, if $v_iv_j\in E(G)$, otherwise $a_{ij}=0$. The spectral radius of  $G$ is the largest eigenvalue of $A(G)$, which is also called $A$ spectral radius of $G$. The degree diagonal matrix  of  $G$ is $D(G)=diag(d_{G}(v_1), d_{G}(v_2),\ldots, d_{G}(v_n))$. The Laplacian matrix of a graph $G$ is $L(G)=D(G)-A(G)$, while the signless Laplacian matrix of $G$ is  $Q(G)=D(G)+A(G)$. The $Q$ spectral radius of  $G$ is the largest eigenvalue of $Q(G)$, while the $L$ spectral radius of  $G$ is the largest eigenvalue of $L(G)$ . In 2017, Nikiforov $\cite{17-Nik}$ defined the $A_\alpha$ matrix of $G$ as $A_\alpha(G)= \alpha D(G) + (1-\alpha)A(G)$ for any real  $\alpha\in[0,1]$, it is clear that $A_0(G)=A(G)$,  $A_{\frac{1}{2}}(G)=\frac{1}{2}Q(G)$ and $A_{1}(G)=D(G)$. Thus the matrix $A_\alpha$ can underpin a unified theory of $A (G)$ and $Q (G)$. 
 The $A_\alpha$ spectral radius of  $G$ is the largest eigenvalue of $A_\alpha(G)$, denoted by $\lambda_{\alpha}(G)$.
 Since $G$ is connected, by the well-known Perron-Frobenius theorem, there exists a
 unit positive eigenvector $X=(x_{v_1},x_{v_2},\ldots,x_{v_n})^T$ corresponding to $\lambda_{\alpha}(G)$. For a vertex $v\in V(G)$, the eigenequation of $A_\alpha(G)$ corresponding to $v$ can be written as:
\begin{align}\label{AX}
\lambda_\alpha(G)x_v=\alpha d_G(v)x_v+(1-\alpha)\sum_{uv\in E(G)}x_u.
\end{align}

One of the classic problems in spectral extremal graph theory is Problem \ref{p-1}, which was raised generally by Brualdi and Solheid \cite{86-BRU} in 1986.
 \begin{problem}\label{p-1}
  Given a set of graphs $\mathcal{G}$, establish an upper or a lower bound for the spectral radius of
  graphs in $\mathcal{G}$, and characterize the graphs in which the maximum or the minimum spectral radius is achieved.
 \end{problem}
Recently, many results about the Problem \ref{p-1} can be refer to \cite{REF1,22nowheel,s-r4, REF6} and the references therein.
  It is similar to Problem \ref{p-1}, for a set of graphs with given some specific properties, how about its $A_\alpha$ spectral radius and what kind of graphs can attach the extremal values of $A_\alpha$  spectral radius?

In 2018, Nikiforov et al. \cite{18-NikO} determined the unique graph of order $n$ with given diameter at least $k$ having the maximum $A_\alpha$ spectral radius. Other results for the $A_\alpha$ spectral radius of some graphs with given  degree sequence, cut vertices and clique number, vertex degrees and so on can be refered to \cite{ 19-LCMdegree, 18-LHX,  REF12} and  the references therein. Further, other results for the $A_\alpha$ spectral radius of digraph can be refered to \cite{Aa-3, Aa-2, Aa-1} and the references therein.

Next, we consider the $A_\alpha$ spectral radius of graphs with given independence number. Let $\mathcal{G}_{n,i}$ be the set of graphs with order $n$ and independence number $i$. For  adjacent spectral radius, in 2009, Xu et al. $\cite{I-10}$ characterized the graphs in $\mathcal{G}_{n,i}$ having the minimum spectral radius with $i \in\{1,2,\lfloor\frac{n}{2}\rfloor\,\lceil\frac{n}{2}\rceil+1,n-3,n-2,n-1\}$.  For the signless Laplacian spectral radius, in 2010, Li et al.  $\cite{I-12}$ characterized the graphs in $\mathcal{G}_{n,i}$ having the minimum signless Laplacian spectral radius with $i \in\{1,2,\lfloor\frac{n}{2}\rfloor\,\lceil\frac{n}{2}\rceil+1,n-3,n-2,n-1\}$. Further, in 2021,  Sun et al.  \cite{SWT-LSC} extended the results to $A_\alpha$ spectral radius for some $\alpha\in[0,1)$. In 2021, Lou and Guo \cite{22-DM} showed that tree has the minimum spectral
radius over all graphs in $\mathcal{G}_{n,i}$ with $i\geq \lfloor\frac{n}{2} \rfloor$, and determined all the graphs in $\mathcal{G}_{n,n-4}$ having the minimum spectral
radius.
Therefore, it is  meaningful to study the $A_\alpha$ spectral radius of the graphs in $\mathcal{G}_{n,i}$ with $i \in[3,\lfloor\frac{n}{2}\rfloor-1]\cup [\lceil\frac{n}{2}\rceil+2,n-4]$.

In this paper, we characterize the graphs having the maximum and the minimum $A_\alpha$ spectral radius with given independence number $n-4$, respectively.
\section{ Preliminaries}
\label{sec:ch-pre}
\textwidth 150mm \textheight 225mm
In this section, we present some preliminary results, which are used in Section 3.
\begin{definition}\label{def-M}\rm(\cite{Bro})
	Let $M$ be a real matrix of order $n$ described in the following block form
	$$
	\left [
	\begin{matrix}
	M_{11} & \cdots & M_{1t} \\
	\vdots & \ddots & \vdots \\
	M_{t1} & \cdots & M_{tt}
	\end{matrix}
	\right  ],
	$$
	where the blocks $M_{ij}$ are $n_i\times n_j$ matrices for any $1\leq i,\ j\leq t$ and $n=n_1 + \cdots + n_t$.  Let $b_{ij}$ be the average row sum of $M_{ij}$ for any $1\leq i,\ j\leq t$. Let $B(M) =(b_{ij})$ be a $t\times t$ matrix (or simply $B$). If any $M_{ij}$ has a constant row sum, then $B$ is called the equitable quotient matrix of $M$.
\end{definition}
\begin{lemma}\label{equ M}\rm(\cite{Ylh})
	Let $B$ be the equitable quotient matrix of $M$ as defined in Definition \ref{def-M}, and $B$ is irreducible. Let $M$ be a nonnegative matrix. Then the spectral radius of $M$ is equal to the spectral radius of $B$.
\end{lemma}
\noindent\begin{lemma}\label{le:subdivide} \rm(\cite{19-LCMdegree}) Let $G$ be a connected graph and $uv$ be an edge in an internal path of $G$. Let $G_{uv}$ be the graph obtained from $G$ by subdividing the edge $uv$ into edges $uw$ and $wv$. Then $\lambda_\alpha(G_{uv})< \lambda_\alpha(G)$ for $\alpha\in[0,1)$.
\end{lemma}
\noindent\begin{lemma}\label{le:g-uv} \rm(\cite{17-Nik}) Let $G$ be a connected graph and $G_0$ be a proper subgraph of $G$. Then $\lambda_\alpha(G_0)<\lambda_{\alpha}(G)$ for $\alpha\in[0, 1)$.
\end{lemma}
\noindent\begin{lemma}\label{le:g-u=g-v} \rm(\cite{17-Nik})
	Let $G$ be a connected graph and $X$ be the positive eigenvector of $A_\alpha(G)$ corresponding to $\lambda_\alpha(G)$. For two vertices $u$, $v$ in $G$, if there exists
	an automorphism $p$: $G \rightarrow G$ such that $p (u) = v$, then $x_u=x_v$ for $\alpha\in[0,1]$.
\end{lemma}
\begin{lemma}\label{le:3.2} \rm(\cite{SWT-LSC})
	Let $G$ be a graph with $u,v\in V(G)$ and there exists
	an automorphism $p$: $G \rightarrow G$ such that $p (u) = v$. Assume $G(s,t)$ ($s\geq t \geq 1$) is a graph obtained from $G$ by attaching $s$ new pendent edges at $u$ and $t$ new pendent edges at $v$. If $\alpha\in[0,1)$, then $\lambda_\alpha(G(s+1,t-1))>\lambda_\alpha(G(s,t))$.
\end{lemma}
\noindent\begin{lemma}\label{le:g-uw+vw} \rm(\cite{18-XLL})
	Let $G$ be a connected graph. For  $u,v\in V(G)$, suppose $Y\subseteq N_G(u)\setminus(N_G(v)\cup\{v\})$. Let $G^{'}=G-\{uw:w\in Y\}+\{vw:w\in Y\}$ and $X$ be a positive eigenvector of $A_{\alpha}(G)$ corresponding to $\lambda_{\alpha}(G)$. If $Y\neq\emptyset$ and $x_v\geq x_u$, then $\lambda_{\alpha}(G^{'})>\lambda_{\alpha}(G)$ for $\alpha\in [0,1)$.
\end{lemma}
\noindent\begin{lemma}\label{le:bound1} \rm(\cite{17-Nik})
Let $G$ be a graph with  maximum degree $\Delta$. If $\alpha\in[0,1)$, then $$\lambda_{\alpha}(G)\geq \frac{1}{2}\big(\alpha(\Delta+1)+\sqrt{\alpha^{2}(\Delta+1)^2+4\Delta(1-2\alpha)}\ \big).$$
{If $G$ is connected, equality holds if and only if $G\cong K_{1,\Delta}$. In particular,}
\begin{equation*}\label{eq:1.1}
\lambda_\alpha(G)\geq
\left\{
\begin{array}{ll}
\alpha(\Delta+1),&  \textrm{if $\alpha\in [0,\frac{1}{2}],$}\\
\alpha\Delta+\frac{(1-\alpha)^2}{\alpha},&  \textrm{if $\alpha\in [\frac{1}{2},1).$}\\
\end{array}
\right.
\end{equation*}
\end{lemma}
\noindent\begin{lemma}\label{le:bound2} \rm(\cite{17-Nik})
	Let $G$ be a graph. If $\alpha\in [0,1]$, then
	$$
	\frac{2|E(G)|}{|V(G)|}\leq \lambda_\alpha(G)\leq \max_{uv\in E(G)}\{\alpha d_G(u)+(1-\alpha)d_G(v)\},
	$$
	the first equality holds if and only if $G$ is regular.
\end{lemma}
\noindent\begin{lemma}\label{Wang1} \rm(\cite{Wang1})
		Let $G$ be a connected graph of order $n$.   Then
		\begin{itemize}
			\item[{\rm (i)}] $\lambda_\alpha(G)<2$ if and only if $G$ is one of the following graphs.
				\begin{itemize}
				\item[{\rm (a)}] $P_n\,(n\geqslant 1)$ for $\alpha\in [0,1),$
				\item[{\rm (b)}] $T_{1,1,n-3}\,(n\geqslant 4)$ for $\alpha\in [0,s_1),$
				\item[{\rm (c)}] $T_{1,2,2}$ for $\alpha\in [0,s_2),$ $T_{1,2,3}$ for $\alpha\in [0, s_3)$ and $T_{1,2,4}$ for $\alpha\in [0, s_4)$.
	        	\end{itemize}
			\item[{\rm (ii)}] $\lambda_\alpha(G)=2$ if and only if $G$ is one of the following graphs.
				\begin{itemize}
				\item[{\rm (a)}] $C_n\,(n\geqslant 3)$ for $\alpha\in [0,1),$
				\item[{\rm (b)}] $P_n\,(n\geqslant 3)$ for $\alpha=1,$
				\item[{\rm (c)}] $W_n\,(n\geqslant 6)$ for $\alpha=0,$
				\item[{\rm (d)}] $T_{1,1,n-3}$ for $\alpha=s_1,$
				\item[{\rm (e)}] $T_{1,2,2}$ for $\alpha=s_2$, $T_{1,2,3}$ for $\alpha=s_3,$ $T_{1,2,4}$ for $\alpha=s_4,$
				\item[{\rm (f)}] $T_{1,3,3}$ for $\alpha= 0,$ $T_{1,2,5}$ for $\alpha=0,$ $K_{1,4}$ for $\alpha=0,$ and $T_{2,2,2}$ for $\alpha=0,$
			\end{itemize}
		\end{itemize}
		where $s_1=\frac{4}{n+1+\sqrt{(n+1)^2-16}}$, $s_2=0.2192+$ is the solution of $2\alpha^3-11\alpha^2+16\alpha-3=0,$ $s_3=0.1206+$ is the solution of $\alpha^3-6\alpha^2+9\alpha-1=0$ and $s_4=0.0517+$ is the solution of $2\alpha^3-13\alpha^2+20\alpha-1=0.$
	\end{lemma}
\section{
 The graphs with the maximum or the minimum $A_\alpha$ spectral radii  among $\mathcal{G}_{n,n-4}$ }
\label{sec:ch-sufficient}
\textwidth 150mm \textheight 225mm
In this section, we character the graphs  having the maximum and the minimum $A_\alpha$ spectral radius among $\mathcal{G}_{n,n-4}$, respectively.
Firstly, for completeness, by using Lemma \ref{le:SWT3}, we obtain the following Theorem \ref{co:3.1} immediately.
\noindent\begin{lemma}\label{le:SWT3} \rm(\cite{SWT-LSC})
		Let $G$ be a graph  in $\mathcal{G}_{n,i}$. If $\alpha\in[0,1)$, then $\lambda_\alpha(G)\leq \lambda_\alpha(K_{i}^{c}\vee K_{n-i})$ with equality if and only if $G\cong K_{i}^{c}\vee K_{n-i}$.
\end{lemma}
\noindent\begin{theorem}\label{co:3.1} Let G be a graph in $\mathcal{G}_{n,n-4}$ having the maximum $A_\alpha$  spectral radius. Then $G\cong K_{n-4}^{c}\vee K_{4}$ for $\alpha \in [0,1)$.
\end{theorem}
 Secondly, by using  Lemmas \ref{equ M}, \ref{le:g-uv}, \ref{le:bound2} and \ref{le:SWT1}, we characterize the graphs in $\mathcal{G}_{n,n-4}$ having the minimum $A_\alpha$ spectral radius with $5\leq n\leq 10$ for  $\alpha \in [0,1)$ in the Theorem \ref{co:3.2}.
\begin{lemma}\label{le:SWT1} \rm(\cite{SWT-LSC})
		Let $G$ be a graph in $\mathcal{G}_{n,i}$ having the minimum $A_\alpha$ spectral radius for $\alpha\in [0,1)$.
		\begin{itemize}
			\item[(i)] If $i=1$, then $G\cong K_n$.
			\item[(ii)] If $i=\lfloor\frac{n}{2}\rfloor$, then $G\cong P_n$ for even $n$ and $G\cong C_n$ for odd $n$.
			\item[(iii)] If $i=\lceil\frac{n}{2}\rceil$, then $G\cong P_n$.
			\item[(iv)] If $i=\lfloor\frac{n}{2}\rfloor+1$ with $n\geq 9$, then $G\cong P_n$ for odd $n$ and $G\cong T_{1,1,n-3}$ for even $n$.
			\item[(v)] If $i=n-1$, then $G\cong K_{1,n-1}$.
		\end{itemize}
	\end{lemma}
\begin{theorem}\label{co:3.2}
	Let $G$ be a graph in $\mathcal{G}_{n,n-4}$  having the minimum $A_\alpha$ spectral radius with $5\leq n\leq 10$ for  $\alpha \in [0,1)$. Let $F_{s,t}\cong (K_{s,t}-e)^c$ where $e=uv$, $u\in V(K_s)$ and $v\in V(K_t)$.
	\begin{itemize}
		\item[\rm(i)] If $n=5$, then $G\cong K_5$ for $\alpha\in[0,1)$.
		\item[\rm(ii)] If $n=6$, then $G\cong F_{3,3}$ for $\alpha\in[0,\frac{7}{9}]$.
		\item[\rm(iii)] If $n=7$, then $G\cong C_7$ for $\alpha\in[0,1)$.
		\item[\rm(iv)] If $n=8,\ 9$, then $G\cong P_n$ for $\alpha\in[0,1)$.
		\item[\rm(v)] If $n=10$, then $G\cong T_{1,1,7}$ for $ \alpha\in[0,1)$.
	\end{itemize}
\end{theorem}
	\begin{proof}	
If $n=5,7,8,9,10$, by using Lemma \ref{le:SWT1}, then we obtain the results of (i), (iii)--(v) immediately.

  If $n=6$, then $G\cong F_{1,5}$, $F_{2,4}$ or $F_{3,3}$.
 Since $F_{2,4}$ is a proper subgraph of $F_{1,5}$,  by using Lemma \ref{le:g-uv}, then $\lambda_{\alpha}(F_{2,4})\leq \lambda_{\alpha}(F_{1,5})$.

  Notice that the partition $V(F_{3,3})=\{u,v\}\cup (V(F_{3,3})\setminus\{u,v\})$ is an equitable partition with an equitable quotient matrix
	$$B(F_{3,3})=
\left [
\begin{matrix}
2\alpha+1 & 2(1-\alpha) \\
1-\alpha & \alpha+1
\end{matrix}
\right  ].
$$
The characteristic polynomial of $B(F_{3,3})$ is
$$f(\alpha,x)=x^2- (3\alpha +2)x + 7\alpha - 1.$$
 Together with Lemma \ref{equ M}, $\lambda_{\alpha}(F_{3,3})$ is the largest root of $f(\alpha,x)$.

 By using Lemma \ref{le:bound2}, $\lambda_{\alpha}(F_{2,4})> \frac{8}{3}$. Obviously,  for  $x>\frac{5}{2}\geq\frac{3\alpha+2}{2}$, $f(\alpha,x)$ is an increasing function. Since
 $$f(\alpha,\frac{8}{3})=-\alpha+\frac{7}{9}>0,\ for\ \alpha\in[0,\frac{7}{9}].$$

 That is
 $$\lambda_{\alpha}(F_{3,3})<\frac{8}{3}<\lambda_{\alpha}(F_{2,4}),\ for\ \alpha\in[0,\frac{7}{9}].$$

 It completes the proof.
\end{proof}

Thirdly, we characterize the graphs in $\mathcal{G}_{n,n-4}$ $(n\geq 11)$ having the minimum $A_\alpha$ spectral radius in Theorem \ref{th:main}. In order to prove Theorem \ref{th:main}, we we shall give some crucial results (Lemmas  \ref{th:ch-10}--\ref{le:3.8}).

\noindent\begin{lemma}\label{th:ch-10}
	Let $G$ be the graph in $\mathcal{G}_{n,n-4}$  with $n\geq 11$ having the minimum $A_\alpha$ spectral radius. Then $G\cong G_{12}(m_1,m_2,m_3,m_4)$ or $G_{13}(m_1,m_2,m_3,m_4)$ in Figure \ref{fig1} for $\alpha\in[0,1)$.	
\end{lemma}
\begin{proof}
Let $S=\{v_1,v_2,v_3,...,v_{n-4}\}$ be the maximum independent set of $G$ and  $H$ be a subgraph of $G$ with $V(H)=\{u_1,u_2,u_3,u_4\}=V(G)\setminus S$. If $n\geq 11$, by using Lemma \ref{le:g-uv}, then $G$ must be a tree, and $H$ is the one graph of  $P_4,\ T_{1,1,1},\ K_1\cup P_3,\ P_2\cup P_2,\ K_1\cup K_1\cup P_2,$ or $K_1\cup K_1\cup K_1\cup K_1$. Therefore, we shall obtain thirteen classes of  trees in $\mathcal{G}_{n,n-4}$ corresponding to the above six structures of the subgraph $H$ in Figure \ref{fig1},  denoted by $G_i(m_1,m_2,m_3,m_4)$ or simply $G_i$,  $i=1,2,\ldots,13$, where $\min\{m_1,m_2,m_3,m_4\}\geq 0$.

\begin{figure}[!h]
\centering
\includegraphics[width=24mm]{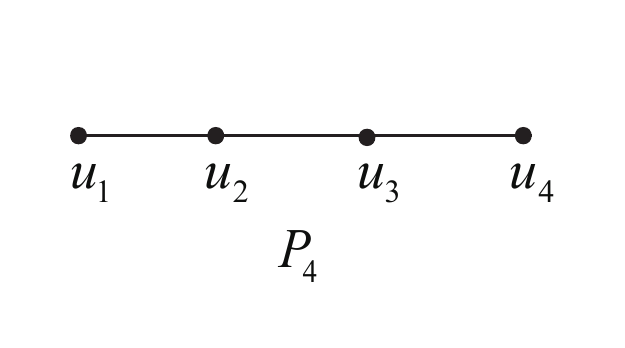}
 \includegraphics[width=24mm]{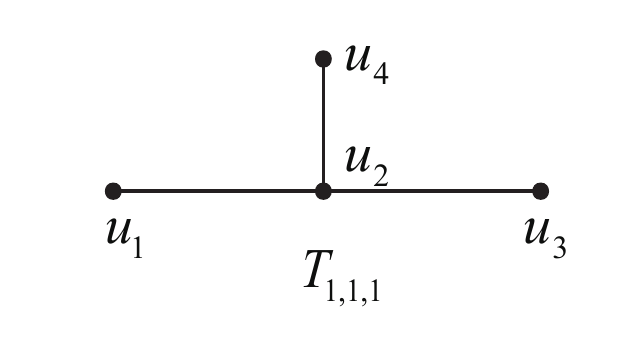}
 \includegraphics[width=24mm]{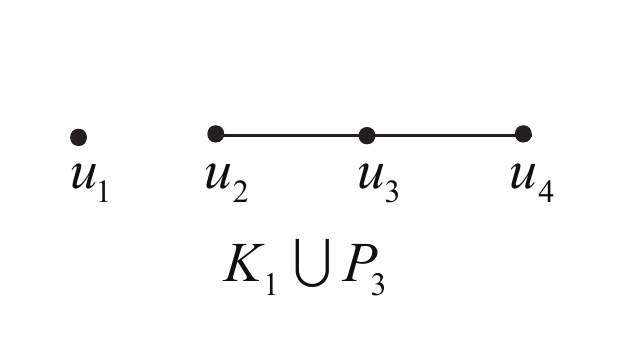}
\includegraphics[width=24mm]{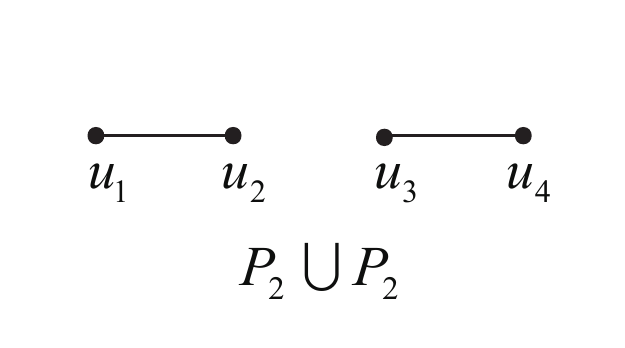}
\includegraphics[width=24mm]{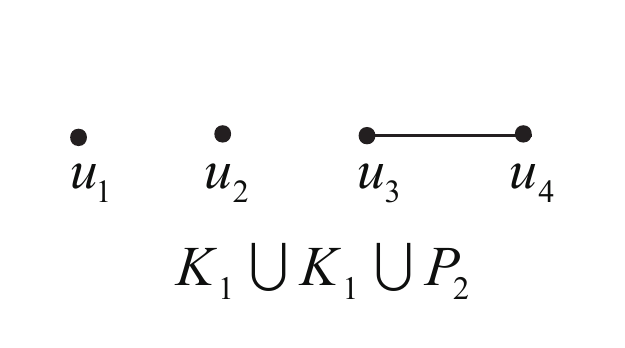}
\includegraphics[width=24mm]{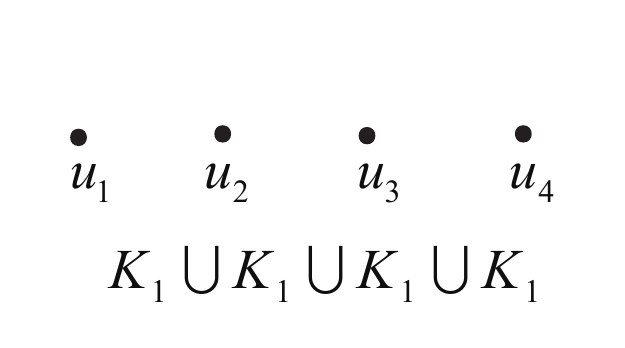}\\
\includegraphics[width=36mm]{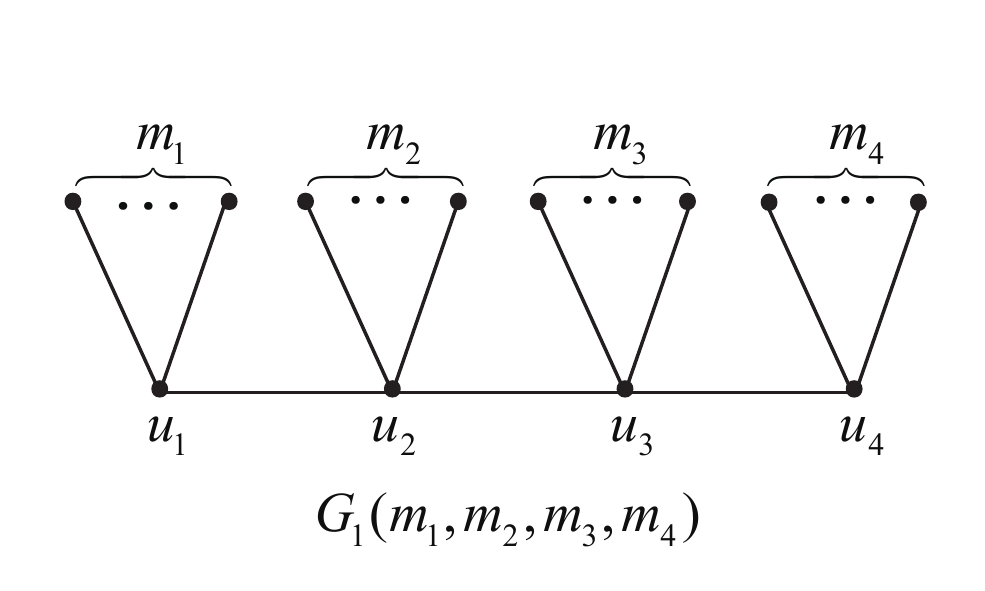}
\includegraphics[width=36mm]{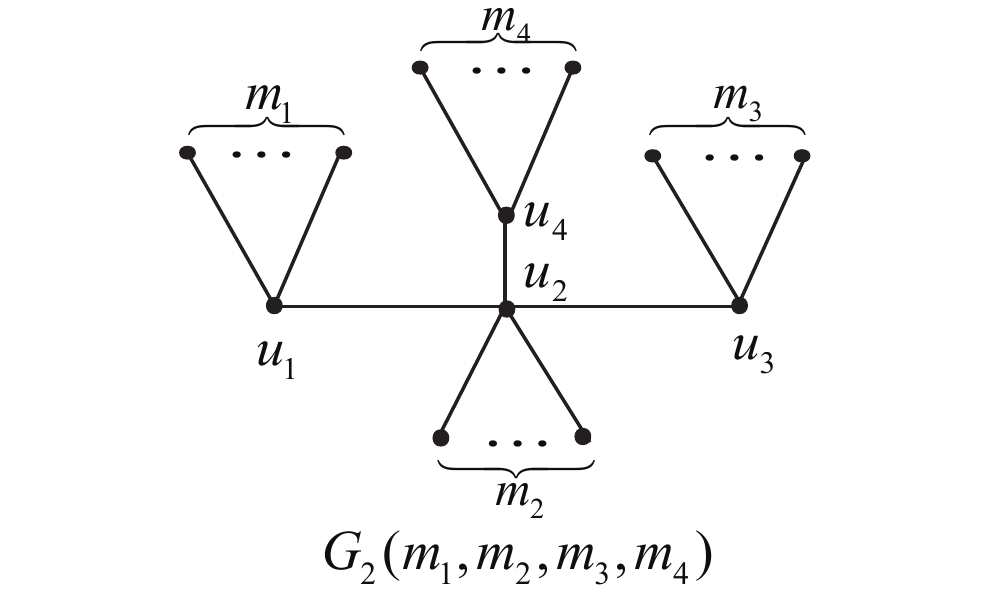}
\includegraphics[width=36mm]{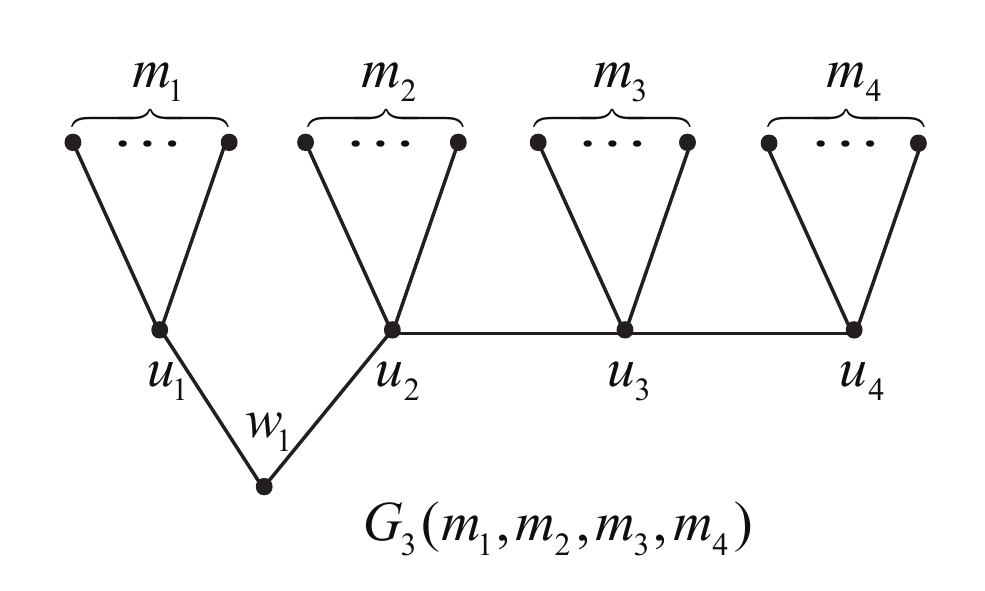}
\includegraphics[width=36mm]{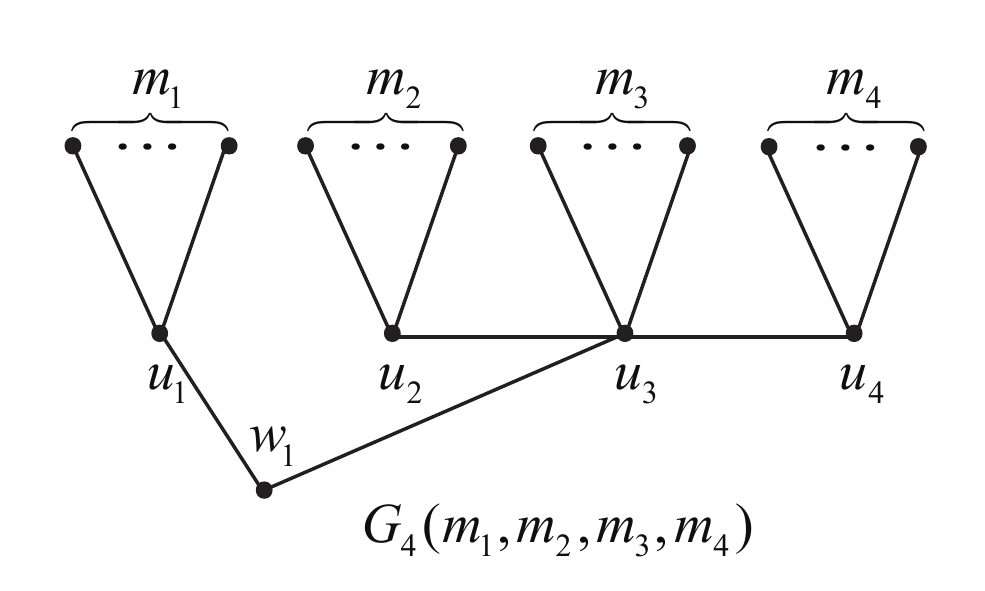}
\includegraphics[width=36mm]{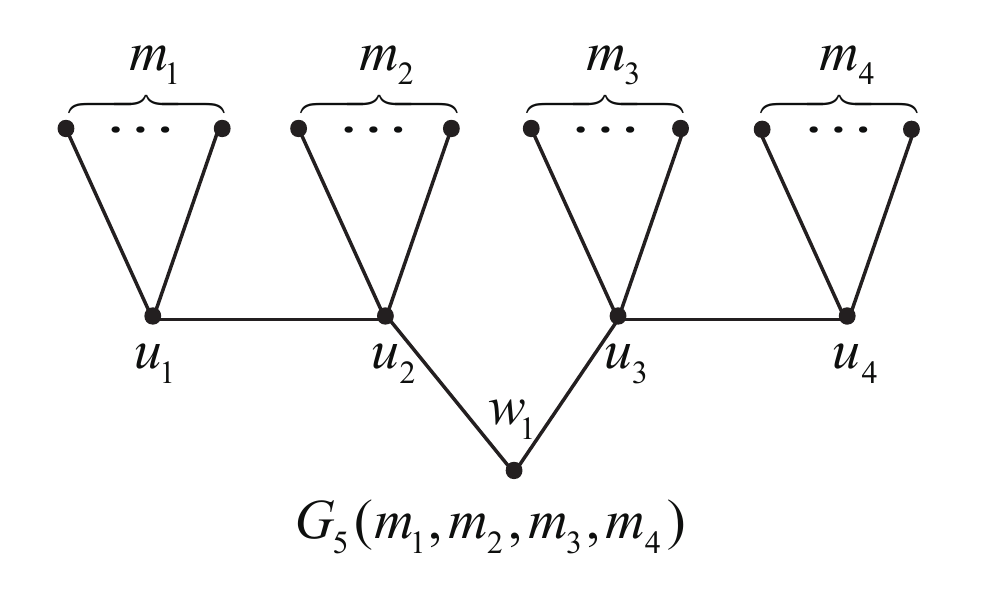}
\includegraphics[width=36mm]{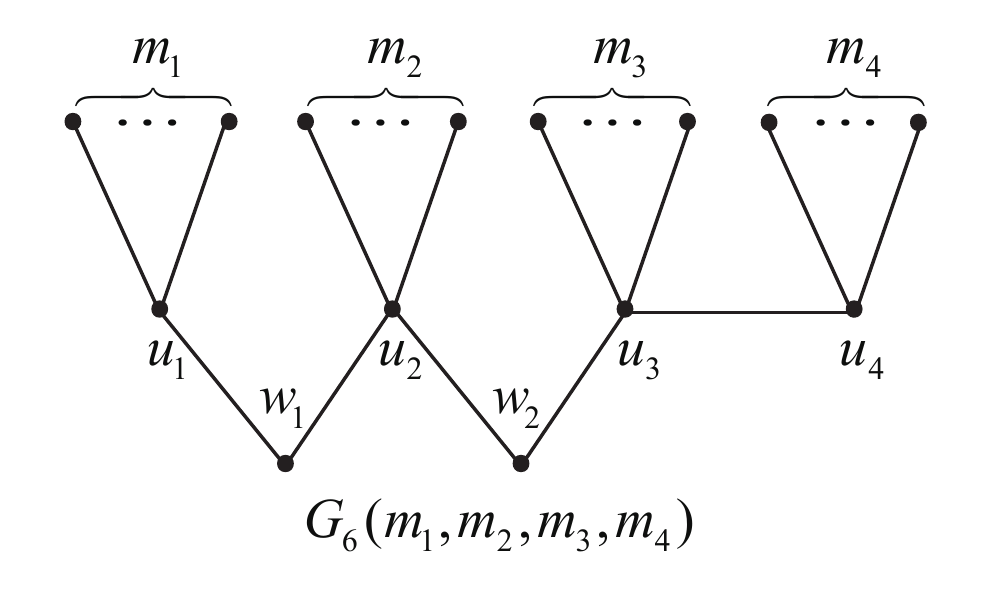}
\includegraphics[width=36mm]{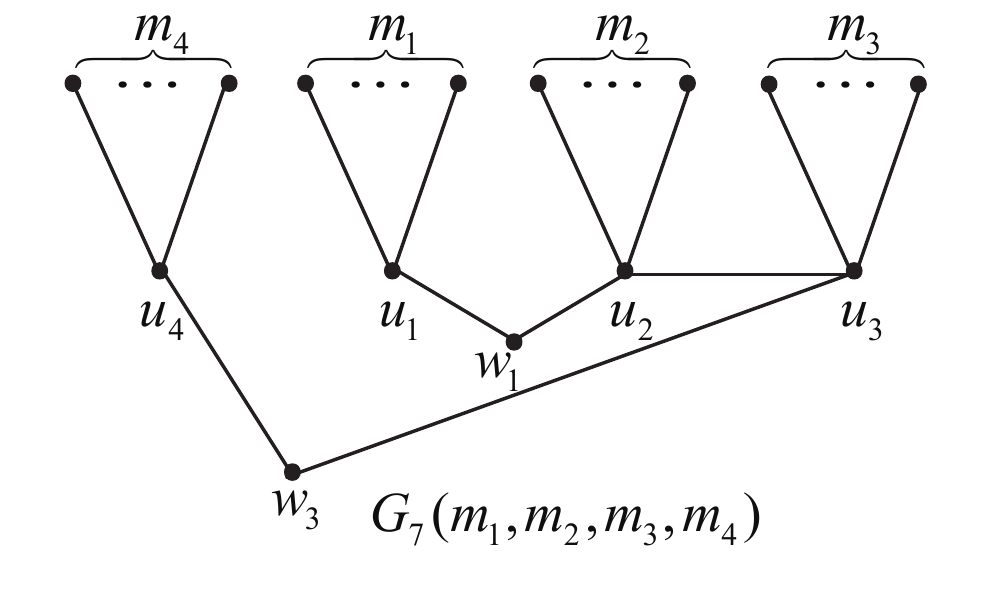}
\includegraphics[width=36mm]{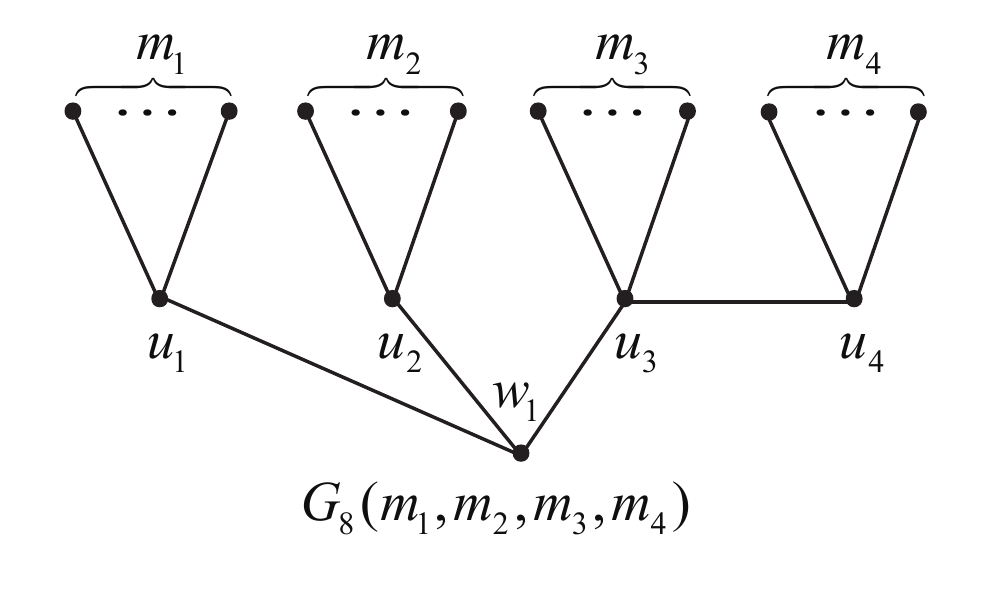}
\includegraphics[width=36mm]{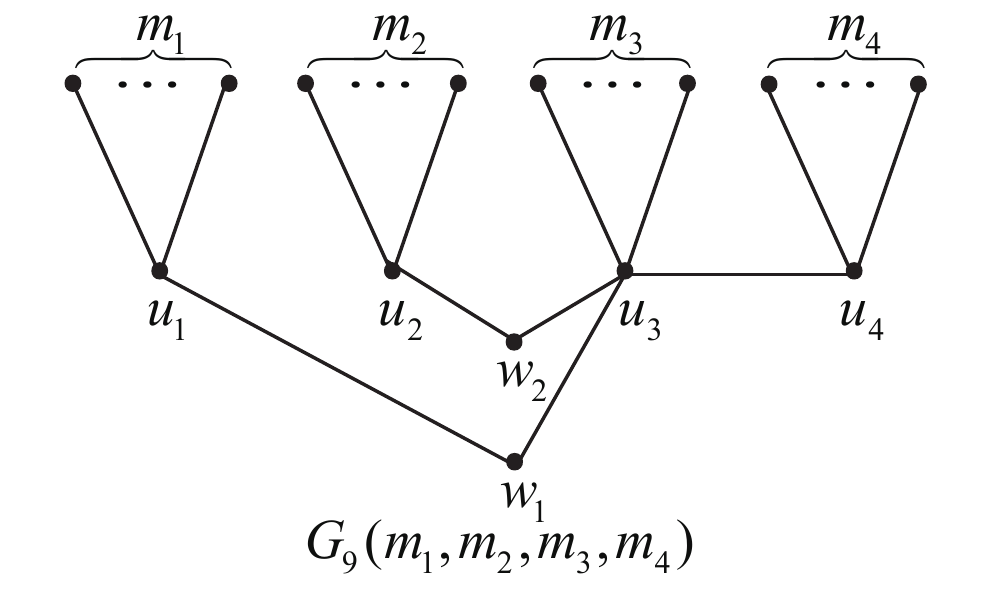}
\includegraphics[width=36mm]{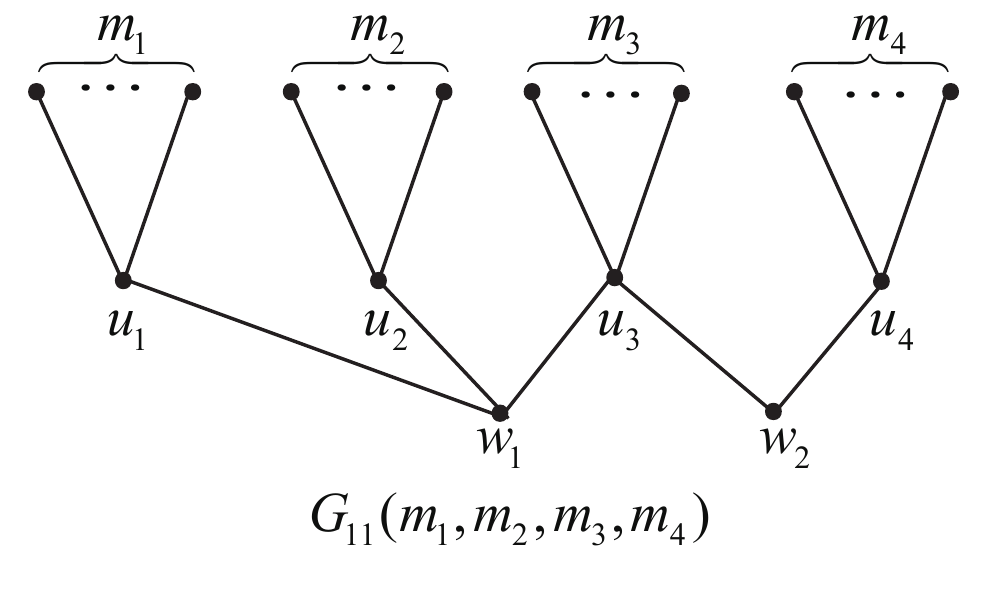}
\includegraphics[width=36mm]{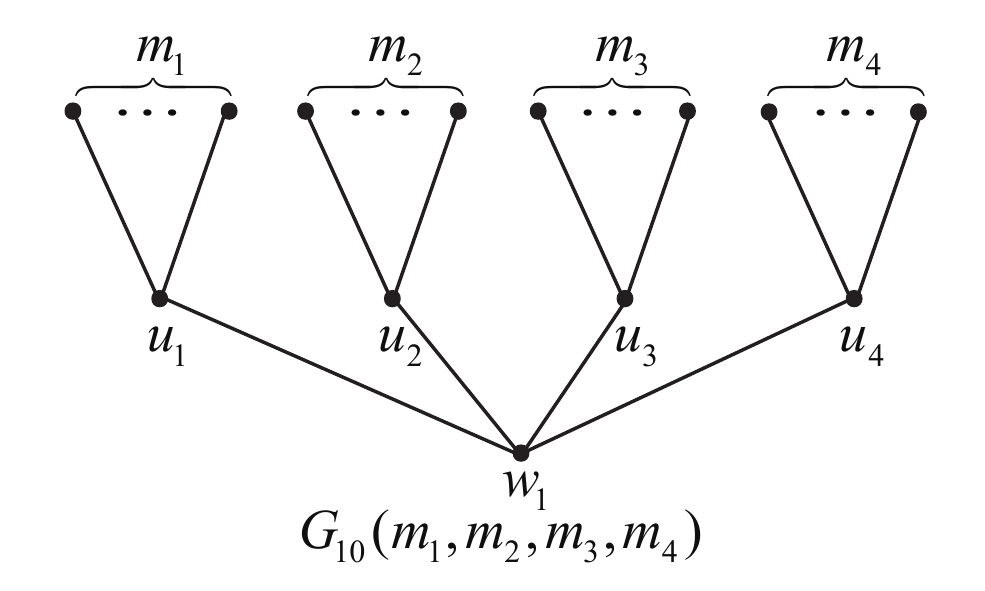}
\includegraphics[width=36mm]{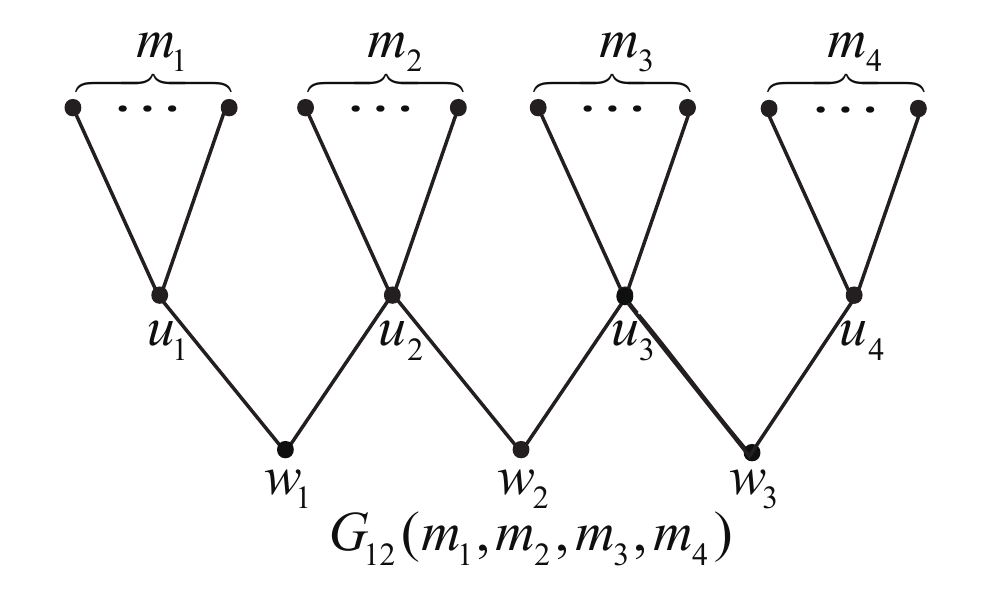}
\includegraphics[width=36mm]{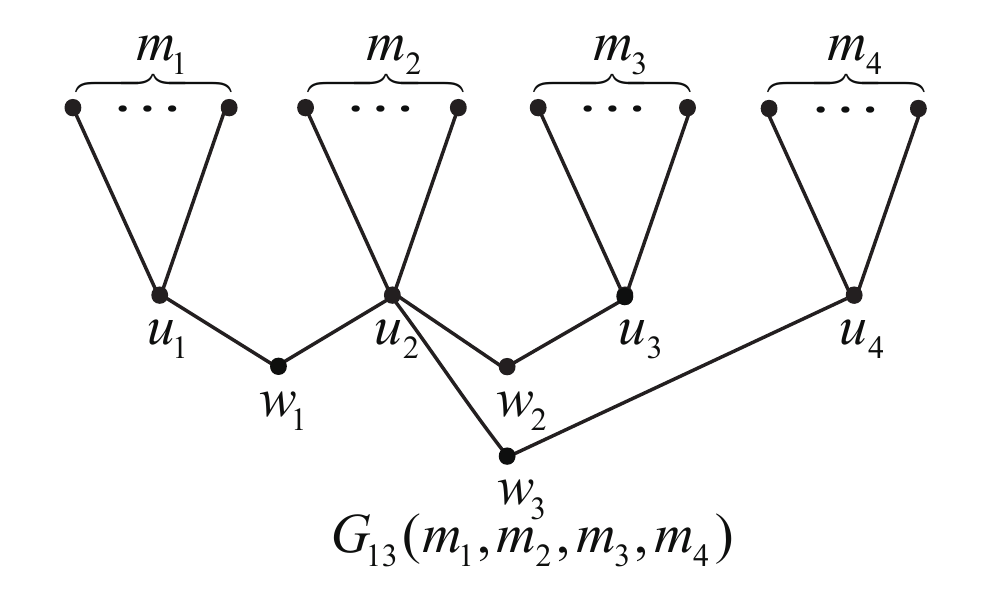}
\caption{Six cases of the subgraph $H$ and their corresponding trees}
\label{fig1}
\end{figure}



We use $G^{(1)}$ to denote a graph obtained from $G$ by subdividing one edge.\\
\noindent\textbf{Case 1.} Suppose $G\cong G_{11}(m_1,m_2,m_3,m_4)$. Let $X=\{x_1,x_2,\ldots,x_n\}^T$ be the positive eigenvector of $A_\alpha(G_{13}(m_1,m_3-1,m_2,m_4))$ corresponding to $\lambda_\alpha(G_{13}(m_1,m_3-1,m_2,m_4))$.
\begin{itemize}
	\item [(i)] If $x_{w_1}\geq x_{w_2}$, by using Lemma \ref{le:g-uw+vw}, then $\lambda_{\alpha}(G_{13}(m_1,m_3-1,m_2,m_4)-u_3w_2+u_3w_1))>\lambda_{\alpha}(G_{13}(m_1,m_3-1,m_2,m_4))$. Notice that $G_{13}(m_1,m_3-1,m_2,m_4)-u_3w_2+u_3w_1\cong G_{11}(m_1,m_2,m_3,m_4)$, then $\lambda_{\alpha}(G_{11}(m_1,m_2,m_3,m_4))>\lambda_{\alpha}(G_{13}(m_1,m_3-1,m_2,m_4))$, which contradicts that $G$ has the minimum $A_\alpha$ spectral radius among all graphs in $\mathcal{G}_{n,n-4}$.
	\item[(ii)] If $x_{w_2}\geq x_{w_1}$, by using Lemma \ref{le:g-uw+vw}, then $
	\lambda_{\alpha}(G_{13}(m_1,m_3-1,m_2,m_4)-u_1w_1+u_1w_2))>\lambda_{\alpha}(G_{13}(m_1,m_3-1,m_2,m_4))$. Notice that $G_{13}(m_1,m_3-1,m_2,m_4)-u_1w_1+u_1w_2\cong G_{11}(m_1,m_2,m_3,m_4)$, then $\lambda_{\alpha}(G_{11}(m_1,m_2,m_3,m_4))>\lambda_{\alpha}(G_{13}(m_1,m_3-1,m_2,m_4))$, a contradiction.
	
	\noindent Thus $G\ncong G_{11}(m_1,m_2,m_3,m_4)$.
\end{itemize}
\noindent\textbf{Case 2.} Suppose $G\cong G_{10}(m_1,m_2,m_3,m_4)$. Let $X=\{x_1,x_2,\ldots,x_n\}^T$ be the positive eigenvector of $A_\alpha(G_{11}(m_1,m_2,m_3-1,m_4))$ corresponding to $\lambda(G_{11}(m_1,m_2,m_3-1,m_4))$.
\begin{itemize}
	\item [(i)] If $x_{w_1}\geq x_{w_2}$,  by using Lemma \ref{le:g-uw+vw}, then $\lambda_{\alpha}(G_{11}(m_1,m_2,m_3-1,m_4)-u_4w_2+u_4w_1))>\lambda_{\alpha}(G_{11}(m_1,m_2,m_3-1,m_4))$. Notice that $G_{11}(m_1,m_2,m_3-1,m_4)-u_4w_2+u_4w_1\cong G_{10}(m_1,m_2,m_3,m_4)$, then $\lambda_{\alpha}(G_{10}(m_1,m_2,m_3,m_4))>\lambda_{\alpha}(G_{11}(m_1,m_2,m_3-1,m_4))$, a contradiction.
	\item[(ii)] If $x_{w_2}\geq x_{w_1}$, by using Lemma \ref{le:g-uw+vw}, then $\lambda_{\alpha}(G_{11}(m_1,m_2,m_3-1,m_4)-u_1w_1-u_2w_1+u_1w_2+u_2w_2))>\lambda_{\alpha}(G_{11}(m_1,m_2,m_3-1,m_4))$. Notice that $G_{11}(m_1,m_2,m_3-1,m_4)-u_1w_1-u_2w_1+u_1w_2+u_2w_2\cong G_{10}(m_1,m_2,m_3,m_4)$, then $\lambda_{\alpha}(G_{10}(m_1,m_2,m_3,m_4))>\lambda_{\alpha}(G_{11}(m_1,m_2,m_3-1,m_4))$, a contradiction.	\\
	Thus $G\ncong G_{10}(m_1,m_2,m_3,m_4)$.
\end{itemize}
\noindent\textbf{Case 3.} Suppose $G\cong G_9(m_1,m_2,m_3,m_4)$.
\begin{itemize}
	\item[(i)] If $m_4=1$, then $G_9(m_1,m_2,m_3,1)\cong G_{13}(m_1,m_3,m_2,0)$.	
	\item[(ii)] If $m_4\geq 2$, then $d_G(u_3),d_G(u_4)\geq 3$ and $u_3u_4$ is an internal path. Let $G_9^{(1)}(m_1,m_2,m_3,\\m_4)$ be a graph obtained from $G_9(m_1,m_2,m_3,m_4)$ by subdividing $u_3u_4$. By using Lemma \ref{le:subdivide}, we have $\lambda_{\alpha}(G_9^{(1)}(m_1,m_2,m_3,m_4))< \lambda_{\alpha}(G_9(m_1,m_2,m_3,m_4))$. \\
	Notice that $G_{13}(m_1,m_3,m_2,m_4-1)$ is a proper subgraph of $G_9^{(1)}(m_1,m_2,m_3,m_4)$, by using Lemma \ref{le:g-uv}, we have $\lambda_{\alpha}(G_{13}(m_1,m_3,m_2,m_4-1))<\lambda_{\alpha}(G_9^{(1)}(m_1,m_2,m_3,m_4))< \lambda_{\alpha}(G_9(m_1,m_2,m_3,m_4))$, a contradiction. \\Thus $G\ncong G_9(m_1,m_2,m_3,m_4)$.
\end{itemize}
\noindent\textbf{Case 4.} Suppose $G\cong G_8(m_1,m_2,m_3,m_4)$.
\begin{itemize}
	\item[(i)] If $m_4=1$, then $G_8(m_1,m_2,m_3,1)\cong G_{11}(m_1,m_2,m_3,0)$, we have obtained that $G\ncong G_{11}$ in Case 1.	
	\item[(ii)] If $m_4\geq 2$, then $d_G(w_1)\geq 3,d_G(u_4)\geq 3$. Thus $u_3u_4$ is an internal path, if $m_3>0$; $w_1u_3u_4$ is an internal path, if $m_3=0$.  By an analogous argument as  Case 3(ii), we can derive a contradiction by subdividing $u_3u_4$ in $G_8$, which is   $\lambda_{\alpha}(G_{11}(m_1,m_2,m_3,m_4-1))<\lambda_{\alpha}(G_8^{(1)}(m_1,m_2,m_3,m_4))< \lambda_{\alpha}(G_8(m_1,m_2,m_3,m_4))$.
\\Thus $G\ncong G_8(m_1,m_2,m_3,m_4)$.
\end{itemize}
\noindent\textbf{Case 5.} Suppose $G\cong G_7(m_1,m_2,m_3,m_4)$.
\begin{itemize}
	\item[(i)] If $m_3=0$, $m_2\neq 0$ and $m_4=1$, then $G_7(m_1,m_2,0,1)\cong G_{12}(m_1,m_2,0,0)$.
	\item[(ii)] If $m_3=0$, $m_2\neq 0$ and $m_4\geq 2$, then $d_G(u_2)\geq 3,d_G(u_4)\geq 3$ and $u_2u_3w_3u_4$ is an internal path. By an analogous argument as Case 3(ii), we can derive a contradiction by subdividing $u_2u_3$ in $G_7$, which is   $\lambda_{\alpha}(G_{12}(m_1,m_2,m_3,m_4-1))<\lambda_{\alpha}(G_7^{(1)}(m_1,m_2,m_3,m_4))< \lambda_{\alpha}(G_7(m_1,m_2,m_3,m_4))$.
	\item[(iii)] If $m_2=m_3=0$, by using Lemma \ref{le:3.2}, we have $\lambda_{\alpha}(G_7(\lfloor\frac{n-6}{2}\rfloor,0,0,\lceil \frac{n-6}{2}\rceil))\leq \lambda_{\alpha}(G_{7}(m_1\\,0,0,m_4))$. Since $m_1+m_2+m_3+m_4+2=n-4$ and $n\geq 11$, then $\lfloor\frac{n-6}{2}\rfloor\geq 2$, thus $u_1w_1u_2u_3w_3u_4$ is an internal path. By an analogous argument as Case 3(ii), we can derive a contradiction by subdividing $u_2u_3$ in $G_7$, which is  $\lambda_{\alpha}(G_{12}(\lfloor\frac{n-6}{2}\rfloor,0,0,\lceil \frac{n-6}{2}\rceil-1))<\lambda_{\alpha}(G_7^{(1)}(\lfloor\frac{n-6}{2}\rfloor,0,0,\lceil \frac{n-6}{2}\rceil))< \lambda_{\alpha}(G_7(\lfloor\frac{n-6}{2}\rfloor,0,0,\lceil \frac{n-6}{2}\rceil))$.
	\item[(iv)] If $m_2\geq 1$ and $m_3\geq 1$, then $d_G(u_2),d_G(u_3)\geq 3$, thus $u_2u_3$ is an internal path. By an analogous argument as Case 3(ii), we can derive a contradiction by subdividing $u_2u_3$ in $G_7$, which is $\lambda_{\alpha}(G_{12}(m_1,m_2,m_3,m_4-1))<\lambda_{\alpha}(G_7^{(1)}(m_1,m_2,m_3,m_4))< \lambda_{\alpha}(G_7(m_1,m_2,m_3,m_4))$.
	\\Thus $G\ncong G_7(m_1,m_2,m_3,m_4)$.
\end{itemize}
\noindent\textbf{Case 6.} Suppose $G\cong G_6(m_1,m_2,m_3,m_4)$.
\begin{itemize}
		\item[(i)] If $m_4=1$,  then $G_6(m_1,m_2,m_3,1)\cong G_{12}(m_1,m_2,m_3,0)$.	
	\item[(ii)] If  $m_4\geq 2$ and $m_3=0$, then $\max \{m_1,m_2\}\geq 1$ (otherwise the independent number of $G_6(0,0,0,m_4)$ is $n-3$). \\	
If $m_2=0$, by an analogous argument as Case 5(iii), then $d_G(u_1),d_G(u_4)\geq 3$ and $u_1w_1u_2w_2u_3u_4$ is an internal path. Thus we can derive a contradiction by subdividing $u_3u_4$ in $G_6$, which is  $\lambda_{\alpha}(G_{12}(\lfloor\frac{n-6}{2}\rfloor,0,0,\lceil \frac{n-6}{2}\rceil-1))<\lambda_{\alpha}(G_6^{(1)}(\lfloor\frac{n-6}{2}\rfloor,0,0,\lceil \frac{n-6}{2}\rceil))< \lambda_{\alpha}(G_6(\lfloor\frac{n-6}{2}\rfloor,0,0,\lceil \frac{n-6}{2}\rceil))<\lambda_{\alpha}(G_6(m_1,0,0,m_4))$.
	
If $m_2\geq 1$, then $d_G(u_2),d_G(u_4)\geq 3$ and $u_2w_2u_3u_4$ is an internal path.
By an analogous argument as Case 3(ii), we can derive a contradiction by subdividing $u_3u_4$ in $G_6$, which is  $\lambda_{\alpha}(G_{12}(m_1,m_2,m_3,m_4-1))<\lambda_{\alpha}(G_6^{(1)}(m_1,m_2,m_3,m_4))< \lambda_{\alpha}(G_6(m_1,m_2,m_3,m_4))$.
	\item[(iii)] If $m_3\geq 1$ and $m_4\geq 2$, then $d_G(u_3),d_G(u_4)\geq 3$ and $u_3u_4$ is an internal paph.  By an analogous argument as Case 3(ii), we can derive a contradiction by subdividing $u_3u_4$ in $G_6$, which is  $\lambda_{\alpha}(G_{12}(m_1,m_2,m_3,m_4-1))<\lambda_{\alpha}(G_6^{(1)}(m_1,m_2,m_3,m_4))< \lambda_{\alpha}(G_6(m_1,m_2,m_3,m_4))$.\\
	Thus $G\ncong G_6(m_1,m_2,m_3,m_4)$.
\end{itemize}
\noindent\textbf{Case 7.} Suppose $G\cong G_5(m_1,m_2,m_3,m_4)$.
\begin{itemize}
\item[(i)] If $m_1=1$, then $G_5(m_1,m_2,m_3,1)\cong G_{6}(0,m_2,m_3,m_4)$,  we have obtained $G\ncong G_6$ in Case 6.
\item[(ii)] If $m_1\geq 2$, $m_2=0$ and $m_3\geq 1$, then $d_G(u_1),d_G(u_3)\geq 3$ and $u_1u_2w_1u_3$ is an internal path.  By an analogous argument as Case 3(ii), we can derive a contradiction by subdividing $u_1u_2$ in $G_5$, which is $\lambda_{\alpha}(G_{6}(m_1-1,m_2,m_3,m_4))<\lambda_{\alpha}(G_5^{(1)}(m_1,m_2,m_3,m_4))\\< \lambda_{\alpha}(G_5(m_1,m_2,m_3,m_4))$.
\item[(iii)] If $m_1\geq 2$ and $m_2\geq 1$, then $d_G(u_1),d_G(u_2)\geq 3$ and $u_1u_2$ is an internal path. By an analogous argument as Case 3(ii), we can derive a contradiction by subdividing $u_1u_2$ in $G_5$, which is $\lambda_{\alpha}(G_{6}(m_1-1,m_2,m_3,m_4))<\lambda_{\alpha}(G_5^{(1)}(m_1,m_2,m_3,m_4))< \lambda_{\alpha}(G_5(m_1,m_2,m_3,m_4))$.
\item[(iv)] If $m_1\geq 2$ and $m_2=m_3=0$, by an analogous argument as Case 5(iii), then $d_G(u_1),d_G(u_4)\geq 3$ and $u_1u_2w_1u_3u_4$ is an internal path. Thus we can derive a contradiction by subdividing $u_1u_2$ in $G_5$, which is  $\lambda_{\alpha}(G_{6}(\lceil \frac{n-5}{2}\rceil-1,0,0,\lfloor\frac{n-5}{2}\rfloor))<\lambda_{\alpha}(G_5^{(1)}(\lceil \frac{n-5}{2}\rceil,0,0,\lfloor\frac{n-5}{2}\rfloor))< \lambda_{\alpha}(G_5(\lceil \frac{n-5}{2}\rceil,0,0,\lfloor\frac{n-5}{2}\rfloor))$.\\Thus $G\ncong G_5(m_1,m_2,m_3,m_4)$.
\end{itemize}
\noindent\textbf{Case 8.} Suppose $G\cong G_4(m_1,m_2,m_3,m_4)$.
\begin{itemize}
	\item[(i)] If $m_2=1$, then $G_4(m_1,1,m_3,m_4)\cong G_{9}(m_1,0,m_3,m_4)$, we have obtained $G\ncong G_9$ in Case 3.
	\item[(ii)] If $m_2\geq 2$, then $d_G(u_2),d_G(u_3)\geq 3$ and $u_2u_3$ is an internal path. By an analogous argument as Case 3(ii), we can derive a contradiction by subdividing $u_2u_3$ in $G_4$, which is $\lambda_{\alpha}(G_{9}(m_1,m_2-1,m_3,m_4))<\lambda_{\alpha}(G_4^{(1)}(m_1,m_2,m_3,m_4))< \lambda_{\alpha}(G_4(m_1,m_2,m_3,m_4))$.
\end{itemize}
\noindent\textbf{Case 9.} Suppose $G\cong G_3(m_1,m_2,m_3,m_4)$.
\begin{itemize}
	\item[(i)] If $m_4=1$, then $G_3(m_1,m_2,m_3,1)\cong G_{7}(m_1,m_2,m_3,0)$, we have obtained $G\ncong G_7$ in Case 5.	
	\item[(ii)] If $m_4\geq 2$ and $\max\{m_2,m_3\}\geq 1$, then $u_3u_4$ ($m_3\geq 1$) or $u_2u_3u_4$ ($m_3=0$) is an internal path. By an analogous argument as Case 3(ii), we can derive a contradiction by subdividing $u_3u_4$ in $G_3$, which is $\lambda_{\alpha}(G_{7}(m_1,m_2,m_3,m_4-1))<\lambda_{\alpha}(G_3^{(1)}(m_1,m_2,m_3,m_4))< \lambda_{\alpha}(G_3(m_1,m_2,m_3,m_4))$.
	\item[(iii)] If $m_4\geq 2$ and $m_2=m_3=0$,  by an analogous argument as Case 5(iii), then $d_G(u_1),d_G(u_4)\geq 3$ and $u_1w_1u_2u_3u_4$ is an internal path. So we can derive a contradiction by subdividing $u_3u_4$ in $G_3$, which is  $\lambda_{\alpha}(G_{7}(\lfloor\frac{n-5}{2}\rfloor,0,0,\lceil \frac{n-5}{2}\rceil-1))<\lambda_{\alpha}(G_3^{(1)}(\lfloor\frac{n-5}{2}\rfloor,0,0,\lceil \frac{n-5}{2}\rceil))< \lambda_{\alpha}(G_3(\lfloor\frac{n-5}{2}\rfloor,0,0,\lceil \frac{n-5}{2}\rceil))<\lambda_{\alpha}(G_3(m_1,0,0,m_4))$.\\
	Thus $G\ncong G_3(m_1,m_2,m_3,m_4)$.
\end{itemize}
\noindent\textbf{Case 10.} Suppose $G\cong G_2(m_1,m_2,m_3,m_4)$.
\begin{itemize}
	\item [(i)] If $m_1=1$, then $G_2(1,m_2,m_3,m_4)\cong G_{4}(0,m_3,m_2,m_4)$, we have obtained  $G\ncong G_4$ in Case 8.
	\item [(ii)] If $m_1\geq 2$, then  $d_G(u_1),d_G(u_2)\geq 3$ and  $u_1u_2$ is an internal path. By an analogous argument as Case 3(ii),  we can derive a contradiction by subdividing $u_1u_2$ in $G_2$, which is $\lambda_{\alpha}(G_{4}(m_1-1,m_3,m_2,m_4))<\lambda_{\alpha}(G_2^{(1)}(m_1,m_2,m_3,m_4))< \lambda_{\alpha}(G_2(m_1,m_2,m_3,m_4))$.\\Thus $G\ncong G_2(m_1,m_2,m_3,m_4)$.
\end{itemize}
\noindent\textbf{Case 11.}
 Suppose $G\cong G_1(m_1,m_2,m_3,m_4)$. \\Then $m_1,m_2,m_3,m_4\geq 1$ and $d_G(u_2),d_G(u_3)\geq 3$, thus $u_2u_3$ is an internal path. By an analogous argument as Case 3(ii),  we can derive a contradiction by subdividing $u_2u_3$ in $G_1$, which is $\lambda_{\alpha}(G_{5}(m_1,m_2-1,m_3,m_4))<\lambda_{\alpha}(G_1^{(1)}(m_1,m_2,m_3,m_4))< \lambda_{\alpha}(G_1(m_1,m_2,m_3,m_4))$.\\
Thus $G\ncong G_1(m_1,m_2,m_3,m_4)$.\\
Therefore, $G\cong G_{12}(m_1,m_2,m_3,m_4)$ or $G_{13}(m_1,m_2,m_3,m_4)$  in Figure \ref{fig1} for $\alpha\in[0,1)$.\\
These complete the proof.
\end{proof}

\noindent\begin{lemma}\label{le:ch-11}
Let $G_1$, $G_2$ are two graphs of order $n$, whose structures are  $G_{12}$ or $G_{13}$ in Figure \ref{fig1}. If $\alpha\in [\frac{1}{2},1)$ and  $\Delta(G_1)-\Delta(G_2)\geq 1$, then $\lambda_{\alpha}(G_1)>\lambda_{\alpha}(G_2)$.
\end{lemma}

\begin{proof}
By Lemmas \ref{le:bound1} and \ref{le:bound2}, if  $\alpha\in[\dfrac{1}{2},1)$, then $\lambda_\alpha(G_1)> \alpha\Delta(G_1)+\frac{(1-\alpha)^2}{\alpha}$, $\lambda_\alpha(G_2)<\alpha\Delta(G_2)+2(1-\alpha)$.

Notice that
$\lambda_\alpha(G_1)-\lambda_\alpha(G_2)>\alpha\Delta(G_1)+\frac{(1-\alpha)^2}{\alpha}-\alpha\Delta(G_2)-2(1-\alpha)=\alpha(\Delta(G_1)-\Delta(G_2))+\frac{1-4\alpha+3\alpha^2}{\alpha}\geq\frac{(2\alpha-1)^2}{\alpha}\geq0$, thus $\lambda_{\alpha}(G_1)>\lambda_{\alpha}(G_2)$.

It completes the proof.
\end{proof}

\begin{lemma}\label{le:3.7}
Let $G\cong G_{12}(m_1,m_2,m_3,m_4)$ in Figure \ref{fig1},  $d=(d_{G}(u_1),d_{G}(u_2),d_{G}(u_3),d_{G}(u_4))$ be the degree subsequence of $G$ satisfying $\max\limits _{1\leq i,j\leq 4}|d_G(u_i)-d_G(u_j)|\leq 2$.
If $m_1\geq 1$ \rm{(resp.} $m_4\geq 1$) and $\Delta(G)> m_1+1$ \rm{(resp.} $\Delta(G)>m_4+1$),  $\alpha\in[\frac{1}{2},1)$, then $x_{u_1}\leq x_{u_2}$ \rm{(resp.} $x_{u_4}\leq x_{u_3}$).

\end{lemma}
\begin{proof}

Notice that $\Delta(G)=\max\{m_1+1,m_2+2,m_3+2,m_4+1\}$, $\sum_{i=1}^{4}m_i=n-7$, and $\max\limits _{1\leq i,j\leq 4}|d_G(u_i)-d_G(u_j)|\leq 2$, then $\sum_{i=1}^{4}d_G(u_i)=n-1$ and $m_1+1\leq \Delta(G)\leq m_1+3$.

Let $V_1= N_G(u_1)\setminus\{w_1\}$, $V_2= N_G(u_2)\setminus\{w_1,w_2\}$ and $V_3= N_G(u_3)\setminus\{w_3\}$. By using Lemma \ref{le:g-u=g-v}, for any two vertices in $V_1$ (resp. $V_2$ or $V_3$), the components of the eigenvector of $A_\alpha(G)$ corresponding to these two vertices are equal, thus we use $x_{v_1}$ (resp. $x_{v_2}$ or $x_{v_3}$) to denote this value.
\begin{equation}\label{eq:6}
\left\{
\begin{array}{c}
\begin{aligned}
\lambda_{\alpha}(G)x_{u_1} &=\alpha(m_1+1)x_{u_1}+(1-\alpha)(m_1x_{v_1}+x_{w_1}),\\
\lambda_{\alpha}(G)x_{u_2}&=\alpha(m_2+2)x_{u_2}+(1-\alpha)(m_2x_{v_2}+x_{w_2}+x_{w_1}),\\
\lambda_{\alpha}(G)x_{u_3}&=\alpha(m_3+2)x_{u_3}+(1-\alpha)(m_3x_{v_3}+x_{w_3}+x_{w_2}),\\
\lambda_{\alpha}(G)x_{u_4}&=\alpha(m_4+1)x_{u_4}+(1-\alpha)(m_4x_{v_4}+x_{w_3}),\\
\lambda_\alpha(G)x_{w_1}&=2\alpha x_{w_1}+(1-\alpha)(x_{u_1}+x_{u_2}),\\
\lambda_\alpha(G)x_{w_2}&=2\alpha x_{w_2}+(1-\alpha)(x_{u_2}+x_{u_3}),\\
\lambda_\alpha(G)x_{w_3}&=2\alpha x_{w_3}+(1-\alpha)(x_{u_3}+x_{u_4}),\\
\lambda_\alpha(G)x_{v_1}&=\alpha x_{v_1}+(1-\alpha)x_{u_1},\\
\lambda_\alpha(G)x_{v_2}&=\alpha x_{v_2}+(1-\alpha)x_{u_2},\\
\lambda_\alpha(G)x_{v_3}&=\alpha x_{v_3}+(1-\alpha)x_{u_3},\\
\lambda_\alpha(G)x_{v_4}&=\alpha x_{v_4}+(1-\alpha)x_{u_4}.\\
\end{aligned}
\end{array}
\right.
\end{equation}

 By using Lemma \ref{Wang1}, it is clear that $\lambda_{\alpha}(G)>2$. Then from (\ref{eq:6}), we have
\begin{equation}\label{eq:7}
\lambda_{\alpha}(G)x_{u_1}=\alpha(m_1+1)x_{u_1}+a_1x_{u_1}+b(x_{u_1}+x_{u_2}),
\end{equation}
{where  $b=\frac{(1-\alpha)^2}{\lambda_{\alpha}(G)-2\alpha}>0$, and  $a_i=\frac{(1-\alpha)^2}{\lambda_{\alpha}(G)-\alpha}m_i>0$ for $i=1,2,3,4$.}

From (\ref{eq:7}), we have
\begin{equation}\label{eq:11}
\frac{x_{u_1}}{x_{u_2}}=\frac{b}{\lambda_{\alpha}(G)-\alpha(m_1+1)-a_1-b}.
\end{equation}

If  $x_{u_1}\geq x_{u_2}$, then $\frac{b}{\lambda_{\alpha}(G)-\alpha(m_1+1)-a_1-b}\geq 1$, it is equal to that $\lambda_{\alpha}(G)-\alpha(m_1+1)-a_1-2b\leq 0$.
Else if  $x_{u_1}\leq x_{u_2}$, then  $\lambda_{\alpha}(G)-\alpha(m_1+1)-a_1-2b\geq 0$.

Let $f(x)=x-\alpha(m_1+1)-a_1-2b$ be a real function of $x$ for $\alpha\in[\frac{1}{2},1)$. That is
\begin{align*}
f(x)(x-\alpha)(x-2\alpha)&=x^3+(-m_1-4)\alpha x^2+((2m_1+3)\alpha^2+(2m_1 + 4)\alpha-m_1-2)x \notag \\
&+(- 4)(m_1 +1)\alpha^2+2(m_1+1)\alpha.
\end{align*}
Let $g(x)=f(x)(x-\alpha)(x-2\alpha)$. By using Lemma \ref{Wang1}, we have $x\geq 2$, thus $(x-\alpha)(x-2\alpha)\geq 0$, then  $g(x)\geq 0$ is equivalent to $f(x)\geq 0$. By calculating the derivative of $g(x)$ with respect to $x$, we have
\begin{align*}
g'(x)&=3x^2-2(m_1+4)\alpha x+(2m_1+3)\alpha^2+2(m_1+2)\alpha-(m_1+2),\\
g''(x)&=6x-2\alpha(m_1+4).
\end{align*}

Since $\alpha\in[\frac{1}{2},1)$ and $m_1+1\leq \Delta(G)\leq m_1+3$, by using Lemma \ref{le:bound1}, then
\begin{align*}
x&\geq \alpha\Delta(G)+\frac{(1-\alpha)^2}{\alpha}>\frac{\alpha\Delta(G)}{3}+\frac{(2\Delta(G)+1)\alpha^2-2\alpha+1}{3\alpha}\\
&\geq \frac{\alpha\Delta(G)}{3}+\frac{5\alpha^2-2\alpha+1}{3\alpha}\geq\frac{\alpha\Delta(G)}{3}+\alpha=\frac{\alpha(\Delta(G)+3)}{3}\\
&\geq \frac{\alpha(m_1+4)}{3}.
\end{align*}

Thus $g''(x)\geq 0$ for $\alpha\in[\frac{1}{2},1)$, so $g'(x)$ is an increasing function for $\alpha\in[\frac{1}{2},1)$. Hence
\begin{align*}
g'(x)&\geq g'(\alpha(m_1+1)+\frac{(1-\alpha)^2}{\alpha})\\
&=\frac{1}{\alpha^2}\big[(m_1^2 + 2m_1 - 1)\alpha^4 + (- 6m_1 - 4)\alpha^3 + (3m_1 + 14)\alpha^2 - 12\alpha + 3\big].
\end{align*}

Let $h(\alpha)=(m_1^2 + 2m_1 - 1)\alpha^4 + (- 6m_1 - 4)\alpha^3 + (3m_1 + 14)\alpha^2 - 12\alpha + 3$ be a real function of $\alpha$ for $\alpha\in[\frac{1}{2},1)$. Then
$$ h'(\alpha)=4\alpha^3(m_1^2 + 2m_1 - 1) - 3\alpha^2(6m_1 + 4) + 2\alpha(3m_1 + 14) - 12,$$
$$ h''(\alpha)=6m_1 + 12\alpha^2(m_1^2 + 2m_1 - 1) - 6\alpha(6m_1 + 4) + 28,$$
$$h'''(\alpha)=24\alpha(m_1^2 + 2m_1 - 1) - 36m_1 - 24.$$
By a direct calculation, we have

\begin{itemize}
\item[(i)] if $m_1\geq 3$ and $\alpha\in[\frac{1}{2},1),$ then $h'''(\alpha)>0,$
\item[(ii)] if $m_1=2$ and $\alpha\in[\frac{24}{42},1),$ then $h'''(\alpha)>0,$
\item[(iii)] if $m_1\leq 1$ and $\alpha\in[\frac{1}{2},1)$, then $h'''(\alpha)<0.$
\end{itemize}
Therefore $h''(\alpha)$ is an increasing function of $\alpha$ for $m_1\geq 3$ and $\alpha\in[\frac{1}{2},1)$. Through a direct calculation, we have

\begin{itemize}
	\item[(i)] if $m_1\geq 3$ and $\alpha\in[\frac{1}{2},1),$ then $h''(\alpha)\geq h''(\frac{1}{2})=3m_1^2-6m_1+13>0$,
   \item[(ii)] if $m_1=2$ and $\alpha\in[\frac{1}{2},1),$ then $h''(\alpha)>0$,
    \item[(iii)]  if $m_1=1$ and $\alpha\in[\frac{1}{2},\frac{15-\sqrt{21}}{12}),$then $h''(\alpha)>0$,
   \item[(vi)]  if $m_1=0$ and $\alpha\in[\frac{1}{2},\frac{\sqrt{30}-3}{3})$, then $h''(\alpha)>0$.
\end{itemize}
So $h'(\alpha)$ is an increasing function of $\alpha$ for $m_1\geq 2$ and $\alpha\in[\frac{1}{2},1)$. Then
\begin{itemize}
\item[(i)]  if $m_1\geq 3$ and $\alpha\in[\frac{1}{2},1),$ then $ h'(\alpha)\geq h'(\frac{1}{2})=\frac{1}{2}(m_1^2-m_1-3)> 0,$
\item[(ii)] if $m_1=2$ and $\alpha\in[0.5391,1)$, then $h'(\alpha)\geq 0$,
\item[(iii)] if $m_1=1$ and $\alpha\in[\frac{3}{4},1)$, then $h'(\alpha)\geq 0$,
\item[(iv)] if $m_1=0$ and $\alpha\in[\sqrt{7}-2,1)$, then $h'(\alpha)\geq 0$.
\end{itemize}

So $h(\alpha)$ is an increasing function of $\alpha$ for $m_1\geq 3$ and $\alpha\in[\frac{1}{2},1)$. Then
\begin{itemize}
\item[(i)]  if $m_1\geq 3$ and $\alpha\in[\frac{1}{2},1)$, then $g'(x)=\frac{1}{\alpha^2}h(\alpha)\geq \frac{1}{\alpha^2}h(\frac{1}{2})=\frac{1}{16\alpha^2}(m_1^2+2m_1-1)> 0$,
\item[(ii)]  if $m_1\in\{0,1,2\}$ and $\alpha\in[\frac{1}{2},1)$, then $g'(x)>0$.
\end{itemize}
Thus $g(x)$ is an increasing function of $x$ for $m_1\geq 0$ and $\alpha\in[\frac{1}{2},1)$.
Further, we consider the following three cases of $\Delta(G)$ by using Lemmas \ref{le:bound1} and \ref{le:bound2}.
\begin{itemize}
\item[(i)]If $\Delta(G)=m_1+1$, then
\begin{align*}
g(x)&\geq g\big(\alpha(m_1+1)+\frac{(1-\alpha)^2}{\alpha}\big)\\
&=\frac{(\alpha-1)^2}{\alpha}\big[(-\alpha^2-2\alpha+1)m_1-\frac{(\alpha-1)^2(2\alpha^2+2\alpha-1)}{\alpha^2}\big].
\end{align*}
Notice that if $m_1\leq \frac{2\alpha^2+2\alpha-1}{-\alpha^3-2\alpha^2+\alpha}$, then $g\big(\alpha(m_1+1)+\frac{(1-\alpha)^2}{\alpha}\big)\geq 0$. By a direct calculation, we have $\frac{2\alpha^2+2\alpha-1}{-\alpha^3-2\alpha^2+\alpha}<0$ for $\alpha\in[\frac{1}{2},1)$. Thus $g\big(\alpha(m_1+1)+\frac{(1-\alpha)^2}{\alpha}\big)<0$ for $m_1\geq 0$ and $\alpha\in [\frac{1}{2},1)$.
 \item[(ii)]If $\Delta(G)=m_1+2$, by a calculation like \textbf{(i)}, if $m_1\geq 1$ and $\alpha\in[\frac{1}{2},1),$ then
\begin{align*}
g(x)&\geq g(\alpha(m_1+2)+\frac{(1-\alpha)^2}{\alpha})\\
&=\alpha^3m_1^2+\frac{3\alpha^4-6\alpha^3+7\alpha^2-4\alpha+1}{\alpha}m_1+\frac{- 4\alpha^5 + 14\alpha^4 - 20\alpha^3 + 15\alpha^2 - 6\alpha + 1}{\alpha^3}\\&>0.
\end{align*}
That is $x_{u_1}\leq x_{u_2}$.
\item[(iii)]If $\Delta(G)=m_1+3$, by a direct calculation, if $m_1\geq 0$ and $\alpha\in[\frac{1}{2},1)$, then
\begin{align*}
g(x)&\geq g\big(\alpha(m_1+3)+\frac{(1-\alpha)^2}{\alpha}\big)\\
&=2\alpha^3m_1+\frac{(11\alpha^4 - 12\alpha^3 + 10\alpha^2 - 4\alpha + 1)}{\alpha}\\&+\frac{(12\alpha^6 - 26\alpha^5 + 37\alpha^4 - 32\alpha^3 + 18\alpha^2 - 6\alpha + 1)}{\alpha^3}>0.
\end{align*}
That is $x_{u_1}\leq x_{u_2}$.

Considering symmetry, if  $m_4\geq 1$ and $\Delta(G)>m_4+1$,  the conclusion is still established that $x_{u_4}\leq x_{u_3}$ for $\alpha\in[\frac{1}{2},1)$.
\end{itemize}

Therefore the proof is completed.
\end{proof}
\begin{lemma}\label{le:3.8}
	Let $k$ be a positive integer $(k\geq 1)$. Let $G^1\cong G_{12}(k-1,k-1,k-1,k-1)$, $G^2\cong G_{12}(k,k-2,k-2,k)$, where $G_{12}(m_1,m_2,m_3,m_4)$ is in Figure \ref{fig1}. Then $\lambda_{\alpha}(G^1)>\lambda_{\alpha}(G^2)$ for $\alpha\in[\frac{1}{2},1)$.
\end{lemma}
\begin{proof}
 Let $V_1= N_{G^1}(u_1)\setminus\{w_1\}$, $V_2= N_{G^1}(u_2)\setminus\{w_1,w_2\}$ and $V_3= N_{G^1}(u_3)\setminus\{w_3\}$. Note that $V(G^1)=\{u_1,u_4\}\cup\{u_2,u_3\}\cup\{V_1,V_4\}\cup\{V_2,V_3\}\cup\{w_1,w_3\}\cup\{w_2\}$ is an equitable partition with an equitable quotient matrix
	$$B(G^1)=
	\left [
	\begin{matrix}
	\alpha k & 0 &(1-a)(k-1) & 0 & 1-\alpha & 0 \\
	 0 & \alpha(k+1) & 0 &(1-a)(k-1) & 1-\alpha & 1-\alpha \\
	1-\alpha & 0 & \alpha & 0 & 0 & 0  \\
	0	& 1-\alpha  &0& \alpha & 0 & 0  \\
	1-\alpha & 1-\alpha  &0& 0 & 2\alpha & 0  \\
	0 & 2(1-\alpha)  &0& 0 & 0 & 2\alpha
	\end{matrix}
	\right  ],
	$$
	It is the same equitable partition for $V(G^2)$ with an equitable matrix
	
	$$B(G^2)=
	\left [
	\begin{matrix}
	\alpha (k+1) & 0 &(1-a)k & 0 & 1-\alpha & 0 \\
     0 & \alpha k & 0 &(1-a)(k-2) & 1-\alpha & 1-\alpha \\
	1-\alpha & 0 & \alpha & 0 & 0 & 0  \\
	0	& 1-\alpha  &0& \alpha & 0 & 0  \\
	1-\alpha & 1-\alpha  &0& 0 & 2\alpha & 0  \\
	0 & 2(1-\alpha)  &0& 0 & 0 & 2\alpha
	\end{matrix}
	\right  ].
	$$
	Let $f_1(\alpha,k,x)$ be the characteristic polynomial of $G^1$, $f_2(\alpha,k,x)$ be the characteristic polynomial of $G^2$.
	\begin{align*}
	f_1(\alpha,k,x)&=x^6 + (- 7\alpha - 2\alpha k)x^5 + (\alpha^2k^2 + 11\alpha^2k + 17\alpha^2 + 4\alpha k + 4\alpha - 2k - 2)x^4 \\
	&+ (- 4\alpha^3k^2- 19\alpha^3k - 18\alpha^3 - 4\alpha^2k^2 - 26\alpha^2k - 14\alpha^2 + 2\alpha k^2 + 13\alpha k + 7\alpha)x^3 \\&+ (4\alpha^4k^2 + 12\alpha^4k+ 10\alpha^4 + 16\alpha^3k^2 + 46\alpha^3k + 12\alpha^3 - 4\alpha^2k^2 - 15\alpha^2k- 6\alpha^2 \\&- 4\alpha k^2 - 8\alpha k + k^2 + 2k)x^2
	+ (- 4\alpha^5k - 3\alpha^5 - 16\alpha^4k^2 - 24\alpha^4k - 6\alpha^4 - 8\alpha^3k^2 \\&- 4\alpha^3k + 11\alpha^3 + 16\alpha^2k^2 + 16\alpha^2k - 8\alpha^2 - 4\alpha k^2 - 4\alpha k + 2\alpha)x + 8\alpha^5k + 4\alpha^5 \\& +16\alpha^4k^2 - 4\alpha^4k - 10\alpha^4 - 16\alpha^3k^2 + 8\alpha^3 + 4\alpha^2k^2 - 2\alpha^2.
	\end{align*}	
	\begin{align*}
	f_2(\alpha,k,x)&=x^6 + (- 7\alpha - 2\alpha k)x^5 + (\alpha^2k^2 + 11\alpha^2k + 17\alpha^2 + 4\alpha k + 4\alpha - 2k - 2)x^4 \\&+ (- 4\alpha^3k^2- 19\alpha^3k - 19\alpha^3 - 4\alpha^2k^2 - 26\alpha^2k - 12\alpha^2 + 2\alpha k^2 + 13\alpha k + 6\alpha)x^3 \\& + (4\alpha^4k^2 + 12\alpha^4k+ 12\alpha^4 + 16\alpha^3k^2 + 46\alpha^3k + 10\alpha^3 - 4\alpha^2k^2 - 15\alpha^2k - 9\alpha^2 \\&- 4\alpha k^2 - 8\alpha k + 4\alpha + k^2+ 2k  - 1)x^2 + (- 4\alpha^5k - 3\alpha^5 - 16\alpha^4k^2 - 24\alpha^4k - 10\alpha^4\\& - 8\alpha^3k^2 - 4\alpha^3k + 21\alpha^3 + 16\alpha^2k^2 + 16\alpha^2k - 16\alpha^2 - 4\alpha k^2 - 4\alpha k + 4\alpha)x\\& + 8\alpha^5k + 4\alpha^5 + 16\alpha^4k^2 - 4\alpha^4k - 10\alpha^4 - 16\alpha^3k^2 + 8\alpha^3 + 4\alpha^2k^2 - 2\alpha^2.
	\end{align*} \vspace{-1cm}
	\begin{align*}
&f_1(\alpha,k,x)-f_2(\alpha,k,x)\\&=x\big[(- \alpha^3 + 2\alpha^2 - \alpha)x^2 + (2\alpha^4 - 2\alpha^3 - 3\alpha^2 + 4\alpha - 1)x - 4\alpha^4 + 10\alpha^3 - 8\alpha^2 + 2\alpha\big].
	\end{align*}
	It is clear that if $x\geq 2\alpha$, then $f_1(\alpha,k,x)-f_2(\alpha,k,x)\leq 0$ for $\alpha\in[\frac{1}{2},1)$. Notice that both $f_1(\alpha,k,x)$ and $f_2(\alpha,k,x)$ are monic polynomials with real roots, then $f_1(\alpha,k,x),\ f_2(\alpha,k,x)\\\geq 0$ for $x\geq x_{max}^1$, $x\geq x_{max}^2$, respectively, where $x_{max}^1,\ x_{max}^2$ are the largest real roots of $f_1(\alpha,k,x),\ f_2(\alpha,k,x)$, respectively. Thus if $f_1(\alpha,k,x)-f_2(\alpha,k,x)\leq 0$ for $x = \lambda_\alpha(G^2)$, then $\lambda_\alpha(G^1) > \lambda_\alpha(G^2)$, together with Lemma \ref{Wang1}, we have $2\alpha<\alpha\Delta+\frac{(1-\alpha)^2}{\alpha}<\lambda_{\alpha}(G^2)<\lambda_{\alpha}(G^1)$.
	
	It completes the proof.
\end{proof}

\noindent\begin{theorem}\label{th:main}
Let $G$ be the extremal graph with the minimum $A_\alpha$ spectral radius in $\mathcal{G}_{n,n-4}$ $(n\geq11)$ for $\alpha\in [\frac{1}{2},1)$. Let $k$ be a positive integer. Then we have the following four results, where $G_{12}(m_1,m_2,m_3,m_4)$ and $G_{13}(m_1,m_2,m_3,m_4)$ are in Figure \ref{fig1}.
\begin{itemize}
	\item[\rm(i)] If $n=4k$, $k\geq 3$, then
	 $G\cong G_{12}(k-2,k-2,k-2,k-1),\
	G_{12}(k-1,k-3,k-2,k-1),\
	G_{13}(k-2,k-3,k-1,k-1),$ or $
	G_{13}(k-1,k-4,k-1,k-1).$
	\item[\rm(ii)] If $n=4k+1$, $k\geq 3$, then
	$G\cong
     G_{12}(k-1,k-2,k-2,k-1)$, or $
     G_{13}(k-1,k-3,k-1,k-1)$.
	\item[\rm(iii)] If $n=4k+2$, $k\geq 3$, then
	$G\cong
	G_{12}(k,k-2,k-2,k-1)$, $
	G_{12}(k-1,k-1,k-2,k-1)$, $
	G_{12}(k,k-1,k-3,k-1)$, $
	G_{12}(k,k-3,k-1,k-1)$, $
	G_{12}(k,k- 2,k-3,k)$, $
	G_{13}(k,k-3,k-1,k-1)$, $
    G_{13}(k-1,k-2,k-1,k-1)$, or $
    G_{13}(k,k-4,k,k-1)$.
	\item[\rm(iv)] If $n=4k+3$, $k\geq 2$, then
	$G\cong
	G_{12}(k,k-1,k-2,k-1)$, $
	G_{12}(k,k-2,k-1,k-1)$, $
	G_{12}(k,k-2,k-2,k)$, $
	G_{13}(k,k-2,k-1,k-1)$, $
	G_{13}(k,k-3,k,k-1)$, or $
	G_{13}(k,k-4,k,k)$.
\end{itemize}
\end{theorem}

\begin{proof}
 Let $d=(d_{G}(u_1),d_{G}(u_2),d_{G}(u_3),d_{G}(u_4))$ be the degree subsequence of $G$. By using Lemmas \ref{th:ch-10} and \ref{le:ch-11}, we can obtain all possible graphs corresponding to the following four cases of $n$ respectively.
\begin{itemize}
	\item[\rm(i)] If $n=4k$ and $k\geq 3$, then $G\cong G_{12}(k-2,k-2,k-2,k-1)$, $G_{12}(k-1,k-3,k-2,k-1)$, $G_{13}(k-2,k-3,k-1,k-1)$ or  $G_{13}(k-1,k-4,k-1,k-1)$ for $\alpha\in[\frac{1}{2},1)$.
\item[\rm(ii)]
If $n=4k+1$ and $k\geq 3$, then $G\cong G_{12}(k-1,k-2,k-2,k-1)$ or $G_{13}(k-1,k-3,k-1,k-1)$ for $\alpha\in[\frac{1}{2},1)$.
\item[\rm(iii)]
If $n=4k+2$ and $k\geq 3$. Let $G_{j}^i\cong G_{j}(m_1^i,m_2^i,m_3^i,m_4^i)$ be the graph in $\mathcal{G}_{n,n-4}$ corresponding to $d_1=(k+1,k,k,k)$, $d_2=(k,k+1,k,k)$, $d_3=(k+1,k+1,k-1,k)$, $d_4=(k+1,k+1,k,k-1)$, $d_5=(k+1,k-1,k+1,k)$, $d_6=(k+1,k,k+1,k-1)$, $d_7=(k+1,k,k-1,k+1)$, $d_8=(k,k+1,k+1,k-1)$, $d_9=(k+1,k+1,k+1,k-2)$, $d_{10}=(k+1,k+1,k-2,k+1) $, respectively, where $i=1,2,\ldots,10$, $j=12,13$. For example, if $i=1$ and $j=12$, then $G_{12}^1(k,k-2,k-2,k-1)$ is  the graph corresponding to $d_1=(k+1,k,k,k)$.\\
By Lemma \ref{le:3.7}, Then
\begin{align*}
\lambda_{\alpha}(G_{12}^1)&<\lambda_{\alpha}(G_{12}^6),\ \lambda_{\alpha}(G_{12}^2)<\lambda_{\alpha}(G_{12}^8)<\lambda_{\alpha}(G_{12}^9),\ \lambda_{\alpha}(G_{12}^3)<\lambda_{\alpha}(G_{12}^4)<\lambda_{\alpha}(G_{12}^9).
\end{align*}

By Lemma \ref{le:3.2},  we have
$$ \lambda_{\alpha}(G_{12}^7)<\lambda_{\alpha}(G_{12}^{10}),\ \lambda_{\alpha}(G_{12}^{8})<\lambda_{\alpha}(G_{12}^{9}),\\
\lambda_{\alpha}(G_{13}^1)<\lambda_{\alpha}(G_{13}^{p}),\
\lambda_{\alpha}(G_{13}^2)<\lambda_{\alpha}(G_{13}^{q}),\\$$
where $p=6,7$, $q=3,4,8,9,10$.

Therefore, if $n=4k+2$ and $k\geq 3$, then $G\cong G_{12}^1$, $G_{12}^2$, $G_{12}^3$, $G_{12}^5$, $G_{12}^7$, $G_{13}^1$, $G_{13}^2$, or $G_{13}^5$ for $\alpha\in[\frac{1}{2},1)$.
\item[\rm(iv)]
If $n=4k+3$ and $k\geq 2$. Let $G_{j}^i\cong G_{j}(m_1^i,m_2^i,m_3^i,m_4^i)$ be the graph in $\mathcal{G}_{n,n-4}$ corresponding to $d_1=(k+1,k+1,k,k)$, $d_2=(k+1,k,k+1,k)$, $d_3=(k+1,k,k,k+1)$, $d_4=(k,k+1,k+1,k)$, $d_5=(k+1,k+1,k+1,k-1)$, $d_6=(k+1,k-1,k+1,k+1)$ respectively, where  $i=1,2,\ldots,6$, $j=12,13$.

By Lemma \ref{le:3.7},  we have
$\lambda_{\alpha}(G_{12}^1)<\lambda_{\alpha}(G_{12}^{5})$.

By Lemmas \ref{le:3.2} and \ref{le:3.8}, we have

\centerline{
$\lambda_{\alpha}(G_{12}^3)<\lambda_{\alpha}(G_{12}^{6})$,
$\lambda_{\alpha}(G_{12}^3)<\lambda_{\alpha}(G_{12}^{4})<\lambda_{\alpha}(G_{12}^5)$, $\lambda_{\alpha}(G_{13}^4)<\lambda_{\alpha}(G_{13}^{5})$.}

Therefore, if $n=4k+3$ and $k\geq 2$, then $G\cong G_{12}^1$, $G_{12}^2$, $G_{12}^3$, $G_{13}^1(\cong G_{13}^4$), $G_{13}^2(\cong G_{13}^3)$, $G_{13}^6$ for $\alpha\in[\frac{1}{2},1)$.
\end{itemize}
These complete the proof.
\end{proof}

\clearpage

\end{document}